\documentclass[a4paper, 11pt]{article}
\usepackage{standalone}

\usepackage[utf8]{inputenc}
\usepackage[T1]{fontenc}
\usepackage[english]{babel}
\usepackage[margin=1in]{geometry}

\usepackage{amsmath}
\usepackage{amsthm}
\usepackage{amsfonts}
\usepackage{amssymb}
\usepackage{mathtools}
\usepackage{tikz}
\usepackage{nicematrix}
\usepackage{leftindex}
\usepackage{faktor}

\usepackage{csquotes}
\usepackage{accents}
\usepackage[babel]{microtype}
\usepackage{enumerate}
\usepackage{lipsum}

\usepackage{graphicx}
\usepackage{xcolor}
\usepackage{hyperref}
\usepackage[textsize = scriptsize, color=green,disable]{todonotes}
\usetikzlibrary{cd,arrows.meta,math,positioning,patterns}
\usepackage{standalone}
\usepackage{bbm}

\usepackage[backend=biber,style=numeric,  eprint=false, sorting=nyt,giveninits=true]{biblatex}
\AtEveryBibitem{%
\ifentrytype{article}{%
\clearfield{url}%
\clearfield{urldate}%
}{}
}
\AtEveryBibitem{%
\ifentrytype{book}{%
\clearfield{url}%
\clearfield{urldate}%
}{}
}

\newcommand{\R}{{\mathbb{R}}}
\newcommand{\C}{{\mathbb{C}}}
\newcommand{\CP}{\mathbb{C}\mathrm{P}}
\newcommand{\Z}{{\mathbb{Z}}}
\newcommand{\Q}{{\mathbb{Q}}}

\DeclareMathOperator{\Span}{Span}
\newcommand{\del}{\partial}
\newcommand{\delbar}{{\bar\partial}}
\newcommand{\iso}{\cong}
\newcommand{\inv}{{-1}}

\DeclareMathOperator{\coker}{coker}
\DeclareMathOperator{\rk}{rk}

\DeclareMathOperator{\Rel}{Rel}

\DeclareMathOperator{\codim}{codim}
\newcommand{\proj}{\mathbb{P}}
\newcommand{\dfn}{\coloneqq}
\newcommand{\nfd}{\eqqcolon}

\AtBeginDocument{
    \let\oldRe\Re
    \let\oldIm\Im
\renewcommand{\Re}{\mathop{\oldRe\mathrm{e}}}
\renewcommand{\Im}{\mathop{\oldIm\mathrm{m}}}
}

\newcommand{\hirz}[1]{\mathbb{F}_{#1}}
\newcommand{\calT}{\mathcal{T}}
\DeclareMathOperator{\cone}{cone}

\newcommand{\trans}[1]{\leftindex^{t}{#1}}

\newcommand{\alg}[1]{\dim_{\mathrm{alg}}\!\left(#1\right)}

\newcommand{\ratcore}[1]{\operatorname{core}_{\Q}(#1)}

\DeclareMathOperator{\LaurMon}{\mathcal{L}}
\newcommand{\FI}{I}

\newcommand{\mer}[1]{\mathcal{M}_{#1}(#1)}

\usepackage{xcolor}

\newcommand*{\vertbar}{\rule[-1ex]{0.5pt}{2.5ex}}
\newcommand*{\horzbar}{\hspace{.5ex}\rule[.5ex]{2.5ex}{0.5pt}\hspace{.7ex}}

\newcommand{\OO}{\mathcal{O}}

\newcommand{\GL}{\mathrm{GL}}

\newcommand{\Conf}{\mathrm{Cf}}
\newcommand{\aff}{\mathbb{A}}

\newcommand{\nicedashrightarrow}[1][1.75em]{\mathrel{
    \tikz[baseline]{\draw[dash pattern=on .25em off .1 em,->](0,.58ex)--(#1,.58ex)}}}
\newcommand{\longdashrightarrow}{\nicedashrightarrow[1.38 em]}
\renewcommand{\dashrightarrow}{\nicedashrightarrow[1 em]}

\let\svthefootnote\thefootnote
\newcommand\freefootnote[1]{%
  \let\thefootnote\relax%
  \footnotetext{#1}%
  \let\thefootnote\svthefootnote%
}

\newtheorem{theorem}{Theorem}
\newtheorem{proposition}{Proposition}[section]
\newtheorem{lemma}[proposition]{Lemma}
\newtheorem*{corollary}{Corollary}

\theoremstyle{definition}
\newtheorem{definition}{Definition}
\newtheorem{remark}[proposition]{Remark}
\theoremstyle{remark}
\newtheorem*{note}{Note}

\newtheoremstyle{named}{}{}{\itshape}{}{\bfseries}{.}{.5em}{#1 \thmnote{#3}}
\theoremstyle{named}
\newtheorem*{namedtheorem}{Theorem}

\newtheoremstyle{named_example}{}{}{}{}{\bfseries}{.}{.5em}{#1Example \thmnumber{#2}\,: \thmnote{#3}}
\theoremstyle{named_example}

\usepackage{zref-clever}
\newcommand{\cref}[1]{\zcref{#1}}
\newcommand{\Cref}[1]{\zcref[S]{#1}}

\newcommand{\labelcref}[1]{\zcref[noname]{#1}}
\newcommand{\namecref}[1]{\zcref[noref]{#1}}

\zcsetup{nameinlink=false,abbrev}

\AddToHook{env/lemma/begin}{%
  \zcsetup{countertype={proposition=lemma}}}
\AddToHook{env/remark/begin}{%
  \zcsetup{countertype={proposition=remark}}}
\AddToHook{env/named_example/begin}{%
  \zcsetup{countertype={proposition=named_example}}}

\zcRefTypeSetup{named_example}{
Name-sg = Example ,
name-sg = example ,
Name-pl = Examples ,
name-pl = examples ,
}

\usepackage[scr=boondox]{mathalpha}
\newcommand{\dd}{\mathscr{d}}

\usepackage[margin=1in]{geometry}

\title{Algebraic Dimension and Reduction of LVMB manifolds}
\author{Federico Thiella\footnote{Dipartimento di Matematica e Informatica, Università degli Studi di Firenze}}
\date{}

\begin{document}
\maketitle
\begin{abstract}
    \noindent
    LVMB manifolds are a large class of compact complex manifolds constructed from an initial combinatorial datum. Most of them are non-algebraic manifolds nor do they support Kähler metrics. All LVMB manifolds are equipped with a canonical holomorphic foliation, whose leaf space is a complex toric quasifold; these spaces constitute the nonrational counterpart of compact toric orbifolds. For an LVMB manifold with no nonzero holomorphic 1-forms, we show that every meromorphic function is a rational function of the Laurent monomials arising from the rational linear relations among the points in the initial datum. It follows that Meersseman’s lower bound for the algebraic dimension is attained and that the algebraic reduction is rational. In particular, both conclusions hold for every LVMB manifold with at most two indispensable indices. Under the same vanishing hypothesis, we relate the algebraic dimension of an LVMB manifold to the nonrationality degree of its canonical quasifold leaf space. We give an explicit formula for this invariant recently introduced by Battaglia and Prato.
    \end{abstract}
\section*{Introduction}
LVMB manifolds are a large class of compact complex manifolds constructed from a geometric--combinatorial datum \((\Lambda,\calT^*)\). This is formed by a list of \(n\) points \(\Lambda\) in a complex affine space of dimension \(m\), and a combinatorial datum \(\calT^*\). The family of LVMB manifolds includes examples that are classical in the literature, such as complex tori, Hopf manifolds \cite{Hopf1948} and Calabi--Eckmann manifolds \cite{Calabi-Eckmann}. While the construction was introduced by S.~López de Medrano and A.~Verjovsky \cite{LopezVerjovsky1997}, their geometry has been deeply investigated by L.~Meersseman \cite{Meersseman2000}. His approach to the construction is essentially based on holomorphic dynamics; a more general, and purely combinatorial construction was subsequently provided by F.~Bosio \cite{Bosio2001}.

One of the peculiarities of LVMB manifolds is that they are neither algebraic nor do they support Kähler metrics, with the only exception of complex tori \cite{Meersseman2000,Bosio2001}. Nevertheless, their geometry is still rich, to the point that it has been extensively investigated by many authors \cite{Meersseman2000, Bosio2001,MeerssemanVerjovsky2004,Cupit-FoutouZaffran2007,PanovUstinovsky2012,BattagliaZaffran2015,Ishida2019, Madera2025,Thiella2026preprint}. However, less is known about their bimeromorphic geometry. Indeed, while it is clear that LVMB manifolds other than complex tori have negative Kodaira dimension \cite{BattistiOeljeklaus2015}, only partial results are available concerning other basic bimeromorphic invariants, such as the algebraic dimension. In this direction, the first result is antecedent to LVMB manifolds: in \cite{LoebNicolau1996} J.~J.~Loeb and M.~Nicolau described the meromorphic functions admitted by a class of compact complex manifolds diffeomorphic to \(S^{2n_1 + 1} \times S^{2n_2+1}\); some of which are LVMB manifolds.

Similarly, in \cite{Meersseman2000} Meersseman provided a lower bound for the algebraic dimension of the subclass of LVM manifolds: his proof carries over identically to the more general LVMB manifolds setting. Essentially, Meersseman shows that the algebraic dimension of an LVMB manifold \(N\) constructed from the initial datum \((\Lambda,\calT^*)\) is bounded below by a constant \(a(\Lambda)\) \cite[Thm.~4]{Meersseman2000} that he defines as the dimension of the space of affine rational relations among the points in \(\Lambda\). Moreover, when the combinatorial datum \(\calT^*\) has no \emph{indispensable indices}, Meersseman's lower bound is attained \cite{Meersseman2000}. However, this hypothesis is quite restrictive, as many interesting examples of LVMB manifolds (e.g. Hopf manifolds or Calabi--Eckmann manifolds) need one or two indispensable indices in their construction. More recently, T.~Panov, Y.~Ustinovsky and M.~Verbitsky proved that, in presence of one or two indispensable indices, the vanishing of \(a(\Lambda)\) implies the vanishing of the algebraic dimension \cite{PanovUstinovskiyVerbitsky2016}.

After a brief review of the construction of LVMB manifolds, tailored to our purposes, we show that, under suitable hypotheses, \(a(\Lambda)\) realizes the algebraic dimension of an LVMB manifold even when its initial configuration presents some indispensable indices.
\begin{namedtheorem}[{\labelcref{algebraic_dimension_formula}}]
    Let \((\Lambda,\calT^*)\) be an LVMB configuration with \(k \geq 1\) indispensable indices. Let \(N = N(\Lambda,\calT^*)\) be the corresponding LVMB manifold. If \(h^{1,0}_\delbar(N)=0\), then \(\alg{N}=a(\Lambda)\).
\end{namedtheorem}
The hypothesis \(h^{1,0}_\delbar(N) = 0\) cannot be omitted in general. In \Cref{Meersseman_counterexample}  we construct an LVM manifold with three indispensable indices and \(a(\Lambda) = 0\) that admits a holomorphic surjection onto an elliptic curve, and hence has positive algebraic dimension.

This never occurs for LVMB manifolds with one or two indispensable indices as in this case the hypothesis \(h^{1,0}_\delbar(N) = 0\) is satisfied automatically. Together with \cite[Thm.~4]{Meersseman2000}, this observation yields:
\begin{corollary}
    Let \(N(\Lambda,\calT^*)\) be an LVMB manifold with \(k \leq 2\) indispensable indices. Then, \(\alg{N} = a(\Lambda)\).
\end{corollary}

Although the difference \(\dim_\C X - \alg{X}\) is a quantitative measure of the “algebraicity defect'' of a compact complex manifold \(X\), it does not provide much information about its field of meromorphic functions. A more refined bimeromorphic invariant is the so-called algebraic reduction \(X^\text{red}\). This is a compact \emph{algebraic} variety---which can in fact be taken to be smooth---with a meromorphic map \(X \dashrightarrow X^\text{red}\) fulfilling the following universal property: for any meromorphic map \(f : X \dashrightarrow A\) to an algebraic variety, there exists a unique rational map \(\tilde f : X^\text{red} \to A\) such that the diagram
\begin{equation*}
    \begin{tikzcd}
        X \ar[rr,dashed] \ar[dr,dashed,"\forall f"'] && X^\text{red}\ar[dl,"\exists! \tilde f",dashed]\\
        & A
    \end{tikzcd}
\end{equation*}
commutes. To obtain such a reduction for an LVMB manifold \(N\), we use the explicit set of generators of \(\mer{N}/\C\) provided by \Cref{algebraic_dimension_formula}. Combining this with \cite[Thm.~4]{Meersseman2000} we obtain the following:
\begin{namedtheorem}[\labelcref{algebraic_reduction}]
    Let \(N = N(\Lambda,\calT)\) be an LVMB manifold with \(k\) indispensable indices. If \(k = 0\) or \(k \geq 1\) and \(h^{1,0}_\delbar(N)= 0\), then the algebraic reduction of \(N\) is birational to \(\CP^{a(\Lambda)}\).
\end{namedtheorem}
The quantity \(a(\Lambda)\) also appears in the topological description of the canonical \(m\)-dimensional foliation \(\mathcal{F}\) that any LVMB manifold carries \cite[\S 2.3]{BattagliaZaffran2015}. One may observe that \(0 \leq a(\Lambda) \leq n - 2m -1\). When \(a(\Lambda) = n - 2m -1\), the canonical foliation of \(N(\Lambda,\calT^*)\) defines a holomorphic Seifert bundle \(\varpi : N(\Lambda,\calT^*) \to X\) over a compact toric variety with at most finite quotient singularities \cite{MeerssemanVerjovsky2004}. In this case, \Cref{algebraic_dimension_formula} yields the following generalization of \cite[Prop~3.2]{Cupit-FoutouZaffran2007}:
\begin{namedtheorem}[{\labelcref{reduction_rational_case}}]
    Let \(N(\Lambda,\calT^*)\) be an LVMB manifold such that \(a(\Lambda) = n - 2m -1\) and \(k \leq 2\). Then the holomorphic Seifert bundle \(\varpi : N(\Lambda,\calT^*) \to X\) is a holomorphic algebraic reduction of \(N(\Lambda,\calT^*)\).
\end{namedtheorem}

When the leaf space \(N/\mathcal{F}\) of the canonical foliation is non-Hausdorff, the link with toric geometry survives: indeed \(N/\mathcal{F}\) always carries the structure of \emph{complex toric quasifold} \cite{BattagliaZaffran2015,BattagliaZaffran2017}. These objects, introduced by E.~Prato in the symplectic setup, generalize toric manifolds to the nonrational setting: a key step in her construction is the replacement of lattices by quasilattices \cite{Prato2001}. For their extension to the complex and algebraic setting see \cite{BattagliaPrato2001,BattagliaPrato2026preprint}. A different viewpoint on this is developed in \cite{KatzarkovLupercioMeerssemanVerjovsky2021}. Recently, it has been shown that one can associate an invariant, called the \emph{nonrationality degree}, to any toric quasifold \cite{BattagliaPrato2026}. 
Namely, the initial datum for a toric quasifold is a triple \((\Sigma,Q,\{v_i\})\), where \(\Sigma\) is a complete simplicial fan in \(\R^d\), \(Q \subset \R^d\) is a \emph{quasilattice}---that is the \(\Z\)-span of a finite set of \(\R\)-generators of \(\R^d\), and \(\{v_i\} \subset Q\) is a distinguished set of generators of the fan rays \cite{Prato2001}. Let \((\overline{Q})^0\) be the vector subspace of \(\R^d\) given by the connected component of \(\overline{Q}\) containing the identity: the nonrationality degree \(\deg_\text{NR}Q\) of the quasilattice, and of the corresponding toric quasifold, is  defined as the dimension of \((\overline{Q})^0\). Moreover, \(Q = L \oplus Q_\text{d}\), where \(Q_\text{d} = Q \cap (\overline{Q})^0\) is dense in \( (\overline{Q})^0\) and \(L\) is a discrete subgroup of \(\R^d\) of rank \(d - \deg_\text{NR}Q\) \cite[Thm.~2]{BattagliaPrato2026}. We will refer to \(d - \deg_\text{NR}Q\) as the \emph{rationality degree} of \(Q\).

In \Cref{rationality_degree_formula} we establish an explicit formula for \(\deg_\text{NR}Q\) in terms of a set of generators of \(Q\) as \(\Z\)-module.
Using Gale duality, we will show that the algebraic dimension of an LVMB manifold coincides with the rationality degree of its leaf space. More precisely:
\begin{namedtheorem}[\labelcref{rat_dimension_is_rat_measure}]
    Let \(N(\Lambda,\calT^*)\) be an LVMB manifold such that \(h^{1,0}_\delbar(N) = 0\). Then its algebraic dimension is the rationality degree of its toric quasifold leaf space. In particular, this holds if \(N\) has less than three indispensable indices.
\end{namedtheorem}

We illustrate this formula on a two-parameter family of generalized Calabi--Eckmann threefolds. We show that the algebraic dimension of these complex manifolds diffeomorphic to \(S^3\times S^3\) takes the values \(0\), \(1\), and \(2\), depending on the rationality of the initial datum.

\paragraph{Acknowledgements}
I am deeply grateful to my advisors Daniele Angella and Fiammetta Battaglia for their constant support during the preparation of this paper.

The author is supported by the Università degli Studi di Firenze, INdAM--GNSAGA and
PRIN 2022 project “Real and Complex Manifolds: Geometry and holomorphic dynamics” (code
2022AP8HZ9).

\section{LVMB manifolds: construction}
We will describe the construction of LVMB manifolds following the approach of Bosio \cite{Bosio2001}, and adopting from the very beginning the convex-geometric framework of Battaglia and Zaffran \cite{BattagliaZaffran2015}.

\subsection{Triangulated vector configurations}
\begin{definition}\label{Triangulated_vector_conf}
    A \emph{triangulated vector configuration} in \(\R^d\) is a pair \((V, \calT)\) where
    \begin{itemize}
        \item \(V = (v_1, \dots, v_n)\) is a list of vectors in \(\R^d\), admitting repetitions. We assume that \(\R^d = \Span_\R (v_1, \dots, v_n)\).
        \item We call \(\tau \subseteq \{1,\dots,n\}\) a \emph{simplex} if \(\{v_j:j\in\tau\}\) is a set of linearly independent vectors. A simplicial \emph{cone} over \(\tau \in \calT\) is the convex set \(\cone(\tau) = \{\sum_{j\in\tau}\R_{\geq 0}v_j\}\). The triangulation \(\calT\) is an abstract simplicial complex of these simplexes, that is:
        \begin{itemize}
            \item \(\cone(\tau) \cap \cone(\tau') = \cone(\tau\cap\tau')\) for all \(\tau,\tau' \in \calT\),
            \item \(\bigcup_{\tau \in \calT} \cone(\tau) \supseteq \cone(V)\).
        \end{itemize}
        In particular \(\cone(\emptyset) = \{0\}\), while, by convention, \(\cone(V) = \sum_{j=1}^n \R_{\geq} v_j\).
    \end{itemize}
\end{definition}
\begin{definition}
    \label{rat_core_and_rat_envelope}
    Let \(W\) be a subspace of \(\R^n\). We say that \(W\) is \emph{rational} if \((W \cap \Q^n) \otimes \R = W\). We call the \emph{rational core} of \(W\) the rational subspace
    \begin{equation*}
        \ratcore{W} = (W \cap \Q^n) \otimes \R = \bigcup_{\substack{T \subseteq W \\ T \text{ rational}}} T.
    \end{equation*}
    On the other hand, we define the \emph{rational envelope} of \(W\) as
    \begin{equation*}
        W^\Q = \bigcap_{\substack{T \supseteq W \\ T \text{ rational}}} T.
    \end{equation*}

    We denote the dimension of these spaces as \(a(W) = \dim \ratcore{W}\) and \(b(W) = \dim W^{\Q}\); clearly \(0 \leq a(W) \leq \dim W \leq b(W) \leq n\).
\end{definition}
We observe that \(W \subseteq \R^n\) is rational if and only if \(\ratcore{W} = W = W^{\Q}\); in this case \(a(W) = \dim W = b(W)\).
\begin{definition}
    \label{rational_relation_def}
    Analogous quantities can be associated to any vector configuration \(V\). For these we set the space of linear relations
    \begin{equation*}
        \Rel(V) = \left\{c \in \R^n : \sum_{j=1}^n c_j v_j = 0\right\} \subseteq \R^n;
    \end{equation*}
    and we define \(a(V) \dfn a(\Rel(V))\) and \(b(V) \dfn b(\Rel(V))\).
\end{definition}
We introduce the following properties for a triangulated vector configuration; these will play a central role in the construction.
\begin{definition}
    \label{properties_of_vector_conf}
    We say that a vector configuration \(V = (v_1,\dots,v_n)\) in \(\R^d\) is
        \begin{itemize}
            \item \emph{odd} if \(n - d = 2m + 1\) for some \(m > 0\),
            \item \emph{graded} if there exists an affine hyperplane \(H \subset E\) such that \(v_j \in H\) for any \(v_j \in V\),
            \item \emph{balanced} if \(\sum_i v_i = 0\).
        \end{itemize}
\end{definition}
\begin{remark}
    The definition of triangulated vector configuration allows some vectors in \(V\) not to belong to any simplex of \(\calT\). These will be called \emph{ghost vectors}, and up to a permutation of the indices, they can be put at the head of \(V\) as \(v_1, \dots, v_k\).
\end{remark}
\begin{remark}
    \label{adding_no_more_than_2_ghosts}
    A triangulated vector configuration \((V,\calT)\) such that \(\cone(V) = \R^d\) can be turned into an odd and balanced one by adding no more than two ghost vectors. Balancedness can be achieved by prepending \(-\sum_i v_i\) to \(V\) and increasing \(n\) by 1; oddness, in turn, can be obtained just by prepending one or two copies of \(0\) to \(V\) and correspondingly increasing \(n\).
\end{remark}

\subsection{Gale Transform}
In \cite{Bosio2001}, LVMB manifolds are constructed from pairs \((\Lambda,\calT^*)\) called ``bons systèmes étudiables''. These are formed by a list \(\Lambda\) of \(n\) points in the complex affine space \(\aff^m_\C\), together with a collection of bases \(\tau^* \in \calT^*\). This is a collection of subsets of \(\{1,\dots,n\}\) fulfilling the so-called Bosio conditions \cite{Bosio2001}. Since we adopt the point of view of triangulated vector configurations, we do not need to specify or verify them, as the construction we are going to employ provides ``bons systèmes étudiables''---which, throughout this paper, we call \emph{LVMB data}---starting from triangulated vector configurations. The transition from triangulated vector configurations to LVMB data is done via the so-called Gale Transform. For the details about this transformation, we refer to \cite{LoeraRambauSantos2010}.
\begin{definition}
    Let \(\Conf_\R(n,d)\) be the set of vector configurations in \(\R^d\) of cardinality \(n\). Let \(V \in \Conf_\R(n,d)\) and \(W \in \Conf_\R(n,n-d)\) be two vector configurations in \(\R^d\) and \(\R^{n-d}\) respectively. We say that \(V\) and \(W\) are in \emph{Gale duality}, or equivalently, that one is a \emph{Gale Transform} of the other, if
    \begin{equation}
        \label{Gale_duality_definition}
        \Rel(V) \overset{\perp}{\oplus} \Rel(W) = \R^n.
    \end{equation}
\end{definition}
Notice that this relation is preserved by any ambient isomorphisms of \(\R^{d}\) and \(\R^{n-d}\); hence the Gale transform depends inherently on an arbitrary choice. Nevertheless, introducing on \(\Conf_\R(n,d)\) the equivalence relation
\begin{equation*}
    V \sim V' \iff \Rel(V) = \Rel(V') \subseteq \R^n,
\end{equation*}
and denoting its equivalence classes by \([V]_\R\), we obtain a well-defined function \(\mathfrak{G} : [V]_\R \mapsto [W]_\R\) that associates to \([V]_\R\) the equivalence class of any configuration Gale dual to \(V\). Since \(\mathfrak{G}^2 = 1\), it defines a bijection \(\Conf_\R(n,d)/_\sim \iso \Conf_\R(n,n-d)/_\sim\).
\begin{remark}
    Since the rationality indices \(a(-)\) and \(b(-)\) in \Cref{rat_core_and_rat_envelope} are well-defined on the classes in \(\Conf_\R(n,d)/_\sim\), we easily obtain from \cref{Gale_duality_definition} that, for any \(V \in \Conf_\R(n,d)\),
    \begin{equation}
        \label{rat_degree_Gale_Dual}
        a([V]_\R) = n - b(\mathfrak{G}([V]_\R)) \quad \text{and} \quad b([V]_\R) = n - a(\mathfrak{G}([V]_\R)).
    \end{equation}
\end{remark}

From \Cref{properties_of_vector_conf}, it is clear that being balanced or graded are properties of vector configurations that descend to their equivalence classes; actually these properties are Gale dual to each other.
\begin{proposition}
    \label{balanced_dual_graded}
    A class of vector configurations \([V]_\R \in \Conf_\R(n,d)/_\sim\) is balanced if and only if \(\mathfrak{G}([V]_\R) \in \Conf_\R(n,n-d)/_\sim\) is graded.
\end{proposition}
\begin{proof}
    Let \(V \in [V]_\R\) be a representative of the class. Denoting by \(\{e_j\}_{j=1}^n\) the canonical basis of \(\R^n\), we observe that \([V]_\R\) is balanced if and only if \(\mathbbm{1}_n \dfn \sum_{j=1}^n e_j \in \Rel(V)\), by definition of Gale duality, this is equivalent to \(\Rel(W) \subseteq \mathbbm{1}_n^\perp\) for any \(W \in \mathfrak{G}([V]_\R)\). Let this be \(W = (w_1,\dots, w_n)\). Consider the linear functional \(\varphi = \sum_{j=1}^n e_j^*\) and the linear map
    \begin{align*}
        \pi : \R^n &\longrightarrow \Span(w_1,\dots,w_n)\\
        e_j &\longmapsto w_j
    \end{align*}
    Since \(\ker\varphi = \mathbbm{1}_n^\perp\) and \(\Rel(W) = \ker\pi\), the inclusion \(\Rel(W) \subseteq \mathbbm{1}_n^\perp\) is equivalent to \(\varphi \circ \ker\pi = 0\). As \(\pi = \coker(\ker\pi)\), this holds if and only if there exists a linear functional \(\psi :\R^{n-d} \to \R\) closing the diagram
    \begin{equation*}
        \begin{tikzcd}
            \empty \ar[r,"\ker\pi"] & \R^n \ar[r,"\pi"] \ar[d,"\varphi"'] & \Span(w_1,\dots,w_n) \ar[dl,dashed,"\exists ! \psi"]\\
            & \R
        \end{tikzcd}
    \end{equation*}
    Therefore we have \(\psi(w_j) = \varphi(e_j) = 1\) for any \(j\), and \(\psi(w_j - w_n) = 0\) for any \(j\). Thus, we have constructed an affine hyperplane \(w_n + \ker\psi\) containing all the vectors in \(W\), hence proving it is graded. Conversely, supposing that \(W\) is graded, we observe that such a hyperplane determines a map \(\psi\) fitting in the diagram above, thus, reversing all the implications in the proof, we conclude that \(V\) is balanced.
\end{proof}
\begin{note}
    In particular, this result shows that Gale duality makes correspond balanced vector configurations in \(\R^d\) to configurations of \emph{points} in \(\aff^{n-d-1}_\R\).
\end{note}

We now return to describing how to obtain an LVMB datum from a triangulated vector configuration; henceforth, we will consider only odd and balanced ones. As we observed in \Cref{adding_no_more_than_2_ghosts}, this is not too restrictive, since these properties can be obtained by adding no more than two ghost vectors to the configuration. For the remainder of this section, all (triangulated) vector configurations considered will consist of \(n\) vectors in \(\R^d\) and will be odd and balanced; in particular, they will satisfy \(n - d = 2m+1 \geq 3\) for some positive integer \(m\).

Given an odd and balanced triangulated vector configuration \((V,\calT)\), we choose an ordered basis \((f_0,\dots, f_{2m})\) for \(\Rel(V)\). Then, by definition, a Gale dual of \(V\) consists of the rows of the matrix
\begin{equation*}
    M =
    \begin{bmatrix}
        \vertbar &  & \vertbar\\
        f_0 & \cdots & f_{2m}\\
        \vertbar &  & \vertbar
    \end{bmatrix}
    =
    \begin{bmatrix}
        \horzbar \hat\Lambda_1^\R \horzbar\\
        \vdots\\
        \horzbar \hat\Lambda_n^\R \horzbar
    \end{bmatrix};
\end{equation*}
the Gale dual configuration we obtain with this choice is \(\hat\Lambda^\R = (\hat\Lambda^\R_1, \dots, \hat\Lambda^\R_{n}) \in \Conf_\R(n,2m+1)\). As \(V\) is balanced, by \Cref{balanced_dual_graded}, \(\hat\Lambda^\R\) is graded, thus we can always choose the \(f_j\)'s such that \(f_0 = \mathbbm{1}_n\). According to this choice, all the vectors \(\hat\Lambda_j^\R\) lie in the affine hyperplane \(H = \{x_1 = 1\}\). Through any linear projection onto \(H\), \(\hat\Lambda^\R\) yields a list \(\Lambda^\R = (\Lambda_1^\R,\dots,\Lambda_n^\R)\) that is a configuration of \emph{points} in the affine space \(\aff_\R^{2m}\). Conversely, any identification of \(\aff^{2m}_\R\) with \(H\) gives rise to a graded vector configuration in \(\R^{2m+1}\). These two processes are mutually inverses modulo the action on \(\R^{2m+1}\) of a linear automorphism of the form
\begin{equation*}
    T = 
    \begin{bmatrix}
        1 &\underline{b}\\
        0 & A
    \end{bmatrix}
\end{equation*}
with \(\underline{b} \in \R^{2m}\) and \(A \in \GL(2m,\R)\). This is a consequence of the non-canonicity of the Gale Transform. Indeed, this is bijective only as a function of the equivalence classes of vector configurations. In particular, the matrix \(M\) itself is defined only up to the right multiplication of a matrix of the form of \(T\); consequently, the configuration of points \(\Lambda^\R\) is defined only up to an affine automorphism of \(\aff_\R^{2m}\).

For configuration of points, we have a natural notion of affine relations:
\begin{definition}
    \label{aff_rel}
    Let \(\Lambda^\R\) be a configuration of \(n\) points in \(\aff^{2m}_\R\). We define the space of their \emph{affine relations} as
    \begin{equation*}
        \Rel_\text{aff}(\Lambda^\R) = \left\{c \in \R^{n} : \sum_{j=1}^n c_j \Lambda^\R_j = 0, \sum_{j=1}^n c_j = 0 \right\}.
    \end{equation*}
\end{definition}
Through the correspondence between graded configurations of vectors and configurations of points discussed above, \(\Rel_\text{aff}(\Lambda^\R)\) is naturally identified with the space \(\Rel(\hat\Lambda^\R)\) of linear relations among the vectors of the corresponding graded configuration. To keep the notation light, we overload it with
\begin{equation*}
    a(\Lambda^\R) \dfn \dim_\R \ratcore{\Rel_\text{aff}(\Lambda^\R)}
\end{equation*}
to denote the dimension of the space of the \emph{affine} rational relations of \(\Lambda^\R\). Since it will always be clear from the context whether a configuration consists of points or vectors, no confusion should be feared.

\subsection{Virtual chamber}
Gale duality acts on the triangulation \(\calT\) of a triangulated vector configuration \((V,\calT)\) as well. Indeed, its dual datum,
\begin{equation*}
    \calT^* = \left\{\{1, \dots, n\}\setminus \tau : \tau \in \calT, |\tau| = d \right\},
\end{equation*}
which is a \emph{virtual chamber} for any choice of Gale dual configuration \(\Lambda^\R\).

\subsection{The construction}
\label{section:TheConstruction}
Gale duality provides a method to construct a pair \((\Lambda^\R,\calT^*)\) from a triangulated vector configuration; the specific representative \(\Lambda^\R \in [\Lambda^\R]_\R\) depends on an arbitrary choice. To turn \(\Lambda^\R\) into a configuration of points \(\Lambda\) in a complex affine space we need to choose an isomorphism between \(\aff^{2m}_\R\) and \(\aff^m_\C\). This is done by selecting a \emph{period matrix}, i.e. a matrix \(\Pi \in \operatorname{Mat}_{m\times 2m}(\C)\) whose real column rank equals \(2m\) (cf. \cite[\S 1.2]{BirkenhakeLange2004}). This is a technical device introduced in the construction of LVMB manifolds to allow greater flexibility \cite{Thiella2026preprint}. Therefore, \(\Lambda^{\R}\) is identified with the \(n\)-tuple of points \(\Lambda = (\Lambda_1, \dots, \Lambda_n)\) in \(\mathbb{A}_\C^m\) given by
\begin{equation*}
    \Lambda =
    \begin{bmatrix}
        \horzbar \Lambda_1 \horzbar\\
        \vdots\\
        \horzbar \Lambda_n \horzbar\\
    \end{bmatrix}
    =
    \begin{bmatrix}
        \horzbar \Lambda_1^{\R} \horzbar\\
        \vdots\\
        \horzbar \Lambda_n^{\R} \horzbar\\
    \end{bmatrix}
    \trans{\Pi}
    = \Lambda^{\R} \trans{\Pi}.
\end{equation*}
\begin{note}
    In the equation above we have deliberately identified a list of row vectors with the matrix having them as rows. As these data contain the same information, no ambiguity arises.
\end{note}
The pair \((\Lambda,\calT^*)\) obtained in this way is an LVMB datum \cite[Prop. 2.1]{BattagliaZaffran2015}. Thus, it allows us to construct a unique LVMB manifold following \cite{Bosio2001}. This is done as follows: the virtual chamber \(\calT^*\) determines an open subset in \(\C^n\) defined as
\begin{equation*}
    U(\calT^*) = \bigcup_{\tau^* \in \calT^*} \bigcap_{j \in \tau^*} D(z_j)
\end{equation*}
where \(D(z_j) = \{z_j \neq 0\}\) denotes the distinguished open subset of the Zariski topology of \(\C^n\). On its projectivization \(\proj(U(\calT^*))\), the datum provided by \(\Lambda\) defines a \(\C^m\)-action via the map
\begin{equation}
    \label{equation:LVMBaction}
    \begin{aligned}
        \alpha : \C^m \times \proj(U(\calT^*)) &\longrightarrow \proj(U(\calT^*))\\
        (u, (z_1 : \dots : z_n)) &\longmapsto \left(e^{2\pi i \left\langle \Lambda_1, u \right\rangle} z_1 : \dots : e^{2\pi i \left\langle \Lambda_n, u \right\rangle} z_n\right)
    \end{aligned}    
\end{equation}
where \(\left\langle \cdot, \cdot\right\rangle\) denotes the standard bilinear pairing in \(\C^m\). The orbit space \(\proj(U(\calT^*))/\alpha\) is a compact complex manifold \cite{Bosio2001} which we will denote by \(N(\Lambda,\calT^*)\).

\begin{remark}
    The action \(\alpha\) is invariant under automorphisms of \(\aff^m_\C\). Thus, if in an LVMB configuration \((\Lambda,\calT^*)\) we replace \(\Lambda\) with \(\Lambda' = \Lambda A + b\) with \(A \in \GL(m,\C)\) and \(b \in \C^m\), the corresponding LVMB manifolds \(N(\Lambda,\calT^*)\) and \(N(\Lambda',\calT^*)\) will be the same. We denote as \([\Lambda]_\C\) the orbits of points configurations under that action.
\end{remark}

\begin{remark}
    \label{indispensableRemark}
    Gale duality is what connects triangulated vector configurations to LVMB data. In particular, as the indices of ghost vectors belong to any \(\tau^* \in \calT^*\), the corresponding coordinates in \(U(\calT^*)\) never vanish. In the literature, these are known as \emph{indispensable indices} \cite{Meersseman2000}. From this perspective, the Gale Transform makes ghost vectors correspond to indispensable points of the LVMB datum. 
\end{remark}
The presence of ghost vectors in a triangulated vector configuration \((V,\calT)\) is reflected by the geometry of \(U(\calT^*)\). Namely,
\begin{proposition}
    \label{splitting_apertone}
    Let \((V,\calT)\) be a triangulated vector configuration and \((V',\calT')\) be a second configuration extending the first with the addition of \(k\) ghost vectors. Assume these to be the first \(k\) vectors in \(V'\). Then, \(U(\calT') = (\C^*)^k \times U(\calT^*)\).
\end{proposition}
\begin{proof}
    Let \(n\) be the number of vectors in the original configuration \((V,\calT)\); consequently, \(n+k\) is the number of vectors in the augmented configuration \((V',\calT')\), the first \(k\) of which are ghosts. Thus, for any \(\tau' \in \calT'\), \(j \in \tau'\) if and only if \(j-k \in \tau\). We can write this shift compactly as \(\tau' = \tau + k\). With this notation, applying Gale Transform, we have
    \begin{equation*}
        \calT'^* = \big\{\{1,\dots,k\} \cup (\tau^* + k) : \tau^* \in \calT^*\big\}.
    \end{equation*}
    Therefore, any \(\tau'^* \in \calT'^*\) can be written as \(\tau'^* = \{1,\dots,k\} \cup (\tau^* + k)\) for a unique \(\tau^* \in \calT^*\). So
    \begin{equation*}
        \bigcap_{j \in \tau'^*} D(z_j) = \bigcap_{j\in \{1,\dots,k\}} D(z_j) \cap \bigcap_{j \in \tau^* + k} D(z_j) = ((\C^*)^k \times \C^{n}) \cap \bigcap_{j \in \tau^* + k} D(z_j) = (\C^*)^k \times \bigcap_{j \in \tau^*+k} D(z_j).
    \end{equation*}
    Taking the union over \(\tau'^* \in \calT^*\), we obtain
    \begin{equation*}
        U(\calT'^*) = \bigcup_{\tau'^* \in \calT'^*} \bigcap_{j \in \tau'^*} D(z_j) = (\C^*)^k \times \bigcup_{\tau^* \in \calT^*} \bigcap_{j \in \tau^*+k} D(z_j) = (\C^*)^k \times U(\calT^*). \qedhere
    \end{equation*}
\end{proof}
In \cite{Meersseman2000}, Meersseman observes that the codimension \(\dd\) of the complement of \(U(\calT^*)\) is linked to the presence of indispensable indices. We report here a proof of the following fact that connects his setup with ours.
\begin{proposition}
    \label{minimal_codimension_no_ghosts}
    \(U(\calT^*) = \C^{n}\setminus S\), where \(S = \bigcup_\ell S_\ell\) is the union of coordinate linear subspaces. If \(\dd = \min_\ell \codim_{\C^n} S_\ell\), then \(\dd > 1\) if and only if \(\calT^*\) does not contain indispensable indices.
\end{proposition}
\begin{proof}
    From the definition of \(U(\calT^*)\), its complement \(S\) can be written as
    \begin{equation*}
        S = \bigcap_{\tau^* \in \calT^*} \bigcup_{j \in \tau^*} V(z_j),
    \end{equation*}
    where each \(V(z_j) = \{z_j = 0\}\) is a Zariski closed subset of \(\C^n\); notice that \(S \neq \emptyset\) as \(0 \in V(z_j)\) for any \(j \in \{1,\dots,n\}\). The proof consists in showing that \(S\) coincides with
    \begin{equation*}
        A = \bigcup_{\tau \notin \calT} \bigcap_{j \in \tau} V(z_j).
    \end{equation*}
    Let \(\underline{z} \in A\); thus, there exists \(\sigma \notin \calT\) such that \(z_j = 0\) for all \(j \in \sigma\). We observe that, for any \(\tau^* \in \calT^*\), the intersection \(\sigma \cap \tau^* \neq \emptyset\); if this were not the case, one would have \(\sigma \subseteq \tau\), which yields the contradiction \(\sigma \in \calT\) since \(\calT\) is closed under descending chains. Thus, there exists \(j \in \sigma \cap \tau^*\) and so, for any \(\tau^* \in \calT^*\), one has \(\prod_{j \in \tau^*} z_j = 0\); thus, \(\underline{z} \in S\).

    Conversely, let \(\underline{z} \in S\). For any \(\tau^* \in \calT^*\) there exists \(j \in \tau^*\) such that \(z_j = 0\). Consider the set
    \begin{equation*}
        J_{\underline{z}} = \bigcup_{\tau^* \in \calT^*} \{j \in \tau^* : z_j = 0\} \nfd \bigcup_{\tau^* \in \calT^*} J_{\tau^*},
    \end{equation*}
    and the intersections \(J_{\tau^*} \cap \tau^*\): these are nonempty since \(z \in S\). Notice that \(J_{\underline{z}} \cap \tau^* \neq \emptyset\) if and only if \(J_{\underline{z}} \notin \calT\). Since the former condition is always satisfied and by construction \(z_j = 0\) for all \(j \in J\), we conclude that \(\underline{z} \in A\). In particular, \(S = A\) implies that \(U(\calT^*)\) can be written in the form
    \begin{equation*}
        U(\calT^*) = \C^n \setminus \bigcup_\ell S_\ell,
    \end{equation*}
    for some coordinate subspaces \(S_\ell\). Notice that, since \(\emptyset \in \calT\), none of these can be the whole of \(\C^n\).
    
    Finally, we observe that if there are no indispensable indices in \(\calT^*\), then \(\{j\} \in \calT\) for any \(j \in \{1,\dots,n\}\). Hence for any \(\tau \notin \calT\)
    \begin{equation*}
        \codim_{\C^n} \bigcap_{j \in \tau} V(z_j) = \sum_{j \in \tau} \codim_{\C^n} V(z_j) = \#\tau > 1,
    \end{equation*}
    and so \(\dd = \min_\ell \codim_{\C^n} S_\ell > 1\). On the other hand, for any indispensable index \(i\) we have \(V(z_i) \subseteq S\); since this is a coordinate hyperplane, we conclude that \(\dd = 1\).
\end{proof}

\section{The algebraic dimension of an LVMB manifold}
A well-known fact about LVMB manifolds is that they admit Kähler metrics only if they are complex tori \cite[Thm.~2]{Meersseman2000}, \cite[Prop.~2.2]{Bosio2001}. Similarly, an LVMB manifold that is bimeromorphic to a projective manifold is necessarily an abelian variety \cite[Thm.~3]{Meersseman2000}, \cite[Prop.~2.2]{Bosio2001}. In literature, compact complex manifolds bimeromorphically equivalent to a projective manifold are called \emph{Moishezon manifolds}; in particular, Moishezon manifolds are the analytic counterparts of smooth algebraic varieties.

Bimeromorphic maps are far more flexible than biholomorphisms: indeed, while \(\CP^1 \times \CP^1\) and \(\CP^2\) are not biholomorphic, they are bimeromorphic via the map
\begin{align*}
    \CP^1 \times \CP^1 &\longdashrightarrow \CP^2\\
    ((a_1:a_2),(b_1:b_2)) &\longmapsto (a_1 b_1 : a_1 b_2 : a_2 b_2)
\end{align*}
that restricts to a biholomorphism onto the image away from the Zariski closed set \(\{a_1 b_2 = 0\}\). Some of the invariants of a compact complex manifold are preserved by bimeromorphisms: one of these is the \emph{algebraic dimension}.
\begin{definition}
    Let \(\mer{X}\) be the field of meromorphic functions of a compact complex manifold \(X\). The algebraic dimension of \(X\) is the transcendence degree of the field extension \(\mer{X}/\C\), and it is denoted as \(\alg{X}\).
\end{definition}
From the definition, it is clear that \(\alg{X}\) is a bimeromorphic invariant. A priori this might be infinite, however, if \(X\) is compact, then \(\alg{X} \leq \dim_{\C}X\) \cite{Siegel1955} (see \cite[Prop.~2.1.9]{Huybrechts2005} for an English translation), with the equality realized by any smooth algebraic variety. The converse holds as well \cite{Moishezon1966} (see \cite[pp. 51--177]{7Papers1967} for an English translation).

In \cite[\S IV]{Meersseman2000}, for any LVMB datum \((\Lambda,\calT^*)\), Meersseman introduces the \(\Q\)-vector space of solutions \(\underline{s} \in \Q^n\) of the linear system
\begin{equation*}
    \begin{cases}
        \sum_{j=1}^n s_j \Lambda_j = 0\\
        \sum_{j=1}^n s_j = 0,
    \end{cases}
\end{equation*}
whose dimension is denoted by \(a\) therein. To be consistent with the notation introduced in \Cref{aff_rel}, we will denote Meersseman's number by \(a(\Lambda)\). This integer is precisely the lower bound for the algebraic dimension of an LVMB manifold established by Meersseman.
\begin{proposition}[{\cite[Thm.~4]{Meersseman2000}}]
    \label{Meersseman_lower_bound}
    Let \(N=N(\Lambda,\calT^*)\) be an LVMB manifold with \(k\) indispensable indices, then \(\alg{N} \geq a(\Lambda)\). Moreover, if \(k=0\), then \(\alg{N} = a(\Lambda)\).
\end{proposition}
\begin{note}
    More precisely, this theorem is proved only for LVM manifolds---a large subclass of LVMB manifolds (see \cite[\S 4]{BattagliaZaffran2015}). Nevertheless, the proof in \cite{Meersseman2000} extends in every part to the more general LVMB setting, hence we treat this as already proven in the form of the statement above.
\end{note}
\begin{remark}
    Any LVMB manifold arising from a configuration \((\Lambda,\calT^*)\) in which all the indices are indispensable is a complex torus \cite{Meersseman2000} (see also the examples in \cite{Thiella2026preprint}). Since in this case \(d = 0\), we have \(a(\Lambda) = 0\). On the other hand, any complex torus---abelian varieties included---can be constructed as an LVMB manifold \cite{Meersseman2000}; thus there is at least one case of an LVMB manifold \(N\) where \(0 = a(\Lambda) < \alg{N}\).
    The fact that the algebraic dimension of a complex torus depends on the \emph{quadratic} relations among the coefficients of its period matrix \cite[\S 2.6]{BirkenhakeLange1999} shows that the algebraic dimension of an LVMB manifold cannot depend solely on the linear information encoded in \((\Lambda,\calT^*)\). Indeed, it is classically known that the algebraic dimension of a complex torus of dimension \(k\) may attain any value between \(0\) and \(k\).
\end{remark}

In the construction of LVMB manifolds, the identification between \(\aff^{2m}_\R\) and \(\aff^m_\C\) induced by the choice of a period matrix yields an isomorphism between their spaces of affine rational relations. This translates into the following.

\begin{proposition}
    \label{a_s_are_equal}
    Let \((\Lambda,\calT^*)\) be any LVMB datum. Then \(a(\Lambda^\R) = a(\Lambda)\).
\end{proposition}
\begin{proof}
    By construction \(\Lambda = \Lambda^\R \trans{\Pi}\) for some period matrix \(\Pi\). Thus, any affine rational relation in \(\Rel(\Lambda^\R) \cap \Q^n\) induces an affine rational relation among the \(\Lambda_j\)'s; this correspondence is clearly injective. Conversely, any \(c \in \Rel_\text{aff}(\Lambda) \cap \Q^{n}\) satisfies \(c \,\Lambda^\R \trans{\Pi} = 0\); taking complex conjugates, we obtain
    \begin{equation*}
        0 = \overline{c \,\Lambda^\R \trans{\Pi}} = c \,\Lambda^\R \trans{\overline{\Pi}},
    \end{equation*}
    hence,
    \begin{equation*}
        c \,\Lambda^\R\, [\trans{\Pi} | \trans{\overline{\Pi}}] = 0.
    \end{equation*}
    Being \(\Pi\) a period matrix, \([\trans{\Pi} | \trans{\overline{\Pi}}] \in \GL(2m,\C)\), whence \(c \, \Lambda^\R = 0\).
\end{proof}

To investigate the structure of the field of meromorphic functions of an LVMB manifold, we introduce the following terminology borrowed from the theory of dynamical systems:
\begin{definition}
    Let \((\Lambda,\calT^*)\) be an LVMB datum. We define the \emph{group of first integrals} of the configuration as the \(\Z\)-module
    \begin{equation*}
        I(\Lambda) = \ratcore{\Rel_\text{aff}(\Lambda)} \cap \Z^n;
    \end{equation*}
    and we refer to its elements simply as \emph{first integrals}.
\end{definition}

We can immediately observe that \(\dim_{\R} \ratcore{\Rel_\text{aff}(\Lambda)} = \rk \FI(\Lambda) = a(\Lambda)\); moreover, by construction, \(\FI(\Lambda)\) is independent of the representative of the class \([\Lambda]_\C\).

Before proving the main result, we need two technical lemmas; the first requires the notion of \emph{asymptotic cone} from convex analysis. For the basic properties of asymptotic cones we refer to the book \cite{Hiriart-UrrutyLemarechal2001}.
\begin{definition}
    Let \(C\) be a non-empty closed convex subset in \(\R^n\); fix \(x \in C\). The \emph{asymptotic cone} of \(C\) is defined as the convex set
    \begin{equation*}
        C_\infty = \{ v \in \R^n : x + tv \in C \text{ for all } t >0\}.
    \end{equation*}
    As \(C\) is convex, this definition does not depend on the choice of \(x \in C\).
\end{definition}
The set \(C_\infty\) encodes all the properties of \(C\) that hold “at infinity''. In particular, the following facts hold:
\begin{itemize}
    \item a closed convex set is compact if and only if \(C_\infty = \{0\}\),
    \item for any affine subspace \(v + W \subset \R^n\), one has \((v+W)_\infty = W\),
    \item for any closed convex cone \(K \subset \R^n\), \(K_\infty = K\).
\end{itemize}
Another remarkable property of asymptotic cones is their compatibility with intersections, namely:
\begin{proposition}[{\cite[Prop 2.2.5]{Hiriart-UrrutyLemarechal2001}}]
    \label{intersection_compatibility}
    Let \(C,D\) be two closed convex sets with \(C \cap D \neq \emptyset\). Then \((C\cap D)_\infty = C_\infty\cap D_\infty\).
\end{proposition}
Although the following \namecref{compact_intersection} can be verified directly, we give a shorter and more streamlined proof by exploiting the properties of asymptotic cones.
\begin{note}
    From now on, we adopt the following convention: given an equivalence class of configurations \([\Lambda]_\C\) of \(n\) points in the affine space \(\aff_\C^m\), we can always choose a representative \(\Lambda\) such that \(\Lambda_1 = 0\). Hence, we can identify such a configuration with the list of vectors \((\Lambda_2,\dots, \Lambda_n)\) in \(\C^m\). Correspondingly, we will always  identify \(\Rel_\text{aff}(\Lambda) \iso \Rel(\Lambda_2, \dots, \Lambda_n)\) with a subspace of \(\R^{n-1}\), and similarly identify \(\FI(\Lambda)\) as a subgroup of \(\Z^{n-1}\).
\end{note}
\begin{lemma}
    \label{compact_intersection}
    Let \((\Lambda,\calT)\) be an LVMB datum with \(k \geq 1\), and let \(0,\Lambda_2,\dots,\Lambda_{k}\) be its indispensable points. We denote as \(\R^{n-k}_{\geq 0}\) the closure of the positive orthant in \(\R^{n-k}\). Then, for any \(J \in \R^{k-1} \times \R^{n-k}_{\geq 0}\), the subset \((J + \Rel(\Lambda)) \cap (\R^{k-1} \times \R^{n-k}_{\geq 0})\) is compact.
\end{lemma}
\begin{proof}
    Both \(J+\Rel(\Lambda)\) and \(\R^{k-1} \times \R^{n-k}_{\geq 0}\) are closed convex subsets of \(\R^{n-1}\); moreover, by assumption, they intersect at \(J\). Thus, taking the asymptotic cones, we have
    \begin{equation}
        \label{equation1}
        (J+\Rel(\Lambda))_\infty \cap (\R^{k-1} \times \R^{n-k}_{\geq 0})_\infty = \Rel(\Lambda) \cap (\R^{k-1} \times \R^{n-k}_{\geq 0})
    \end{equation}
    as \(J+\Rel(\Lambda)\) is an affine subspace and \(\R^{k-1} \times \R^{n-k}_{\geq 0}\) a closed convex cone. 
    
    By the geometry of the LVMB datum \(\Lambda\), the right-hand side of \cref{equation1} is \(\{0\}\).  Indeed, by Gale duality, \(\Rel(\Lambda)\) is spanned by the rows of the matrix
    \begin{equation*}
        \begin{bmatrix}
            \vertbar && \vertbar\\
            v_1 & \cdots & v_{n-1}\\
            \vertbar && \vertbar
        \end{bmatrix},
    \end{equation*}
    with \(\left(-\sum_{j=1}^{n-1} v_j, v_1,\dots, v_{n-1}\right)\) forming an odd and balanced vector configuration that is Gale dual to \(\Lambda^\R\). Thus, any \(w \in \Rel(\Lambda)\) is of the form \(w = \sum_{\ell=1}^{n-1} \langle v_\ell,v\rangle e_\ell\) for some \(v \in \R^d\). If \(\langle v_j, v\rangle\) were all non-negative (or all non-positive) for \(j=k,\dots,n-1\), the vectors \(v_k,\dots, v_{n-1}\) would lie in a closed half plane of \(\R^d\). This never happens, since the non-ghost vectors of a triangulated vector configuration span the rays of a complete fan.

    The statement then follows observing that, by \Cref{intersection_compatibility},
    \begin{equation*}
         ((J+\Rel(\Lambda)) \cap ( \R^{k-1} \times \R^{n-k}_{\geq 0}))_\infty = (J+\Rel(\Lambda))_\infty \cap (\R^{k-1} \times \R^{n-k}_{\geq 0})_\infty = \{0\},
    \end{equation*}
    whence \((J+\Rel(\Lambda)) \cap ( \R^{k-1} \times \R^{n-k}_{\geq 0})\) is compact.
\end{proof}
The second lemma we need strengthens \cite[Lemme~2.2]{LescureMeersseman2002}.
\begin{lemma}
    \label{holomorphic_one_forms_rank}
    Let \(N=N(\Lambda,\calT^*)\) be an LVMB manifold with \(k\geq1\) indispensable indices. After choosing \(\Lambda_1=0\), let \(L\) be the matrix whose rows are the indispensable points \(\Lambda_2,\dots,\Lambda_k\). Then
    \begin{equation*}
        H^0(N,\Omega_N^1) \iso \{c\in\C^{k-1}:cL=0\},
        \qquad h^{1,0}_\delbar(N)=k-1-\operatorname{rk}_{\C}L.
    \end{equation*}
    In particular, \(h^{1,0}_\delbar(N)=0\) if and only if \(L:\C^m\to\C^{k-1}\) is surjective.
\end{lemma}
\begin{proof}
    Let \(\pi : \proj(U(\calT^*)) \to N\) be the projection on the orbit space of the action \(\alpha\) defined in \Cref{section:TheConstruction}. Since the first index is indispensable, \(\zeta_j = \frac{z_j}{z_1}\), for \(j=2,\dots,n\), are coordinates on the affine chart \(\{z_1 \neq 0\} \subset \CP^{n-1}\) containing \(\proj(U(\calT^*))\). Since the holomorphic vector fields \(\zeta_j\frac{\partial}{\partial \zeta_j}\) are constant along the flow of \(\alpha\), they descend to holomorphic vector fields \(Z_j\) on \(N\).

    For any \(\theta\in H^0(N,\Omega_N^1)\), the functions \(f_j=\theta(Z_j)\) are holomorphic on \(N\), hence constant as it is compact. Moreover, for any non-indispensable index \(j>k\), there exists a point \(p = (1: z_2 : \dots: z_n) \in \proj(U(\calT^*))\) with \(z_j=0\). Since at this point the vector field \(\zeta_j\frac{\partial}{\partial \zeta_j}\) vanishes, so does \(Z_j\) at \(\pi(p)\); consequently \(f_j=0\) for \(j > k\). On the open dense subset of \(\proj((\C^*)^{n-1}) \subseteq \proj(U(\calT^*))\) where all coordinates are nonzero, the vector fields \(\zeta_j\frac{\partial}{\partial \zeta_j}\) are a local frame for the tangent bundle; therefore,
    \begin{equation*}
        \pi^*\theta=\sum_{j=2}^n f_j\frac{d\zeta_j}{\zeta_j} =\sum_{j=2}^k f_j\frac{d\zeta_j}{\zeta_j}.
    \end{equation*}
    The last expression is holomorphic on all of \( \proj(U(\calT^*))\), since the indispensable coordinates never vanish. Moreover, the fundamental vector fields of \(\alpha\)
    \begin{equation*}
        \xi_j = 2\pi i\sum_{\ell=2}^{n} \Lambda_{\ell}^j \zeta_\ell\frac{\partial}{\partial \zeta_\ell} \quad j= 1, \dots, m
    \end{equation*}
    are annihilated by \(\pi^* \theta\) as this 1-form is basic. Since \(\Lambda_j = L_{j-1}\) for \(j = 2, \dots, k\), the vanishing of \(\pi^*\theta(\xi_j)\) is precisely the condition \((f_2,\dots,f_k)L=0\).

    Conversely, for any \(c=(c_2,\dots,c_k)\) with \(cL=0\), the form \(\sum_{j=2}^k c_j\frac{d\zeta_j}{\zeta_j}\) is holomorphic and \(\alpha\)-invariant, hence it descends to a holomorphic 1-form on \(N\). It is nonzero whenever \(c\neq0\), as can be seen on the open subset where all coordinates are nonzero. This proves the claimed isomorphism and the dimension formula.
\end{proof}
\begin{remark}
    \label{Dolbeault_indipspensable_points}
    By Bosio's conditions, the indispensable points \((\Lambda_1 = 0, \Lambda_2,\dots, \Lambda_k)\) are vertices of a simplex; in particular, they cannot be all zero. Hence, if \(k \geq 2\) then \(\rk_\C L \geq 1\). On the other hand, if \(k = 1\) the matrix \(L\) has no rows. This implies that \(h^{1,0}_\delbar(N)=0\) holds on every \(N\) with one or two indispensable points.
\end{remark}

We can now state the main result of the paper.
\begin{theorem}
    \label{algebraic_dimension_formula}
    Let \((\Lambda,\calT^*)\) be an LVMB configuration with \(k \geq 1\) indispensable indices. Let \(N = N(\Lambda,\calT^*)\) be the corresponding LVMB manifold. If \(h^{1,0}_\delbar(N)=0\), then \(\alg{N}=a(\Lambda)\). In particular, for any LVMB manifold with \(k \leq 2\), one has \(\alg{N} = a(\Lambda)\). 
\end{theorem}
\begin{proof}
    The proof consists of several steps:
    \paragraph{Step 1}
    By construction, \(N\) is the quotient of \(\proj(U(\calT^*))\) by the \(\C^m\)-action \(\alpha\) defined in \cref{equation:LVMBaction}. If \(\pi : \proj(U(\calT^*)) \to N\) is the holomorphic projection on the orbit space, we can observe that the inverse image sheaf \(\pi^{-1}\mathcal{M}_N\) is the subsheaf of \(\mathcal{M}_{\proj(U(\calT^*))}\) consisting of local sections that are \(\alpha\)-invariant meromorphic functions on \(\proj(U(\calT^*))\); we denote it by \(\mathcal{M}_{\proj(U(\calT^*))}^\alpha\). Moreover, since \(\pi\) is a surjective holomorphic submersion with connected fibers, we have \(\mathcal{M}_N \iso \pi_*\mathcal{M}_{\proj(U(\calT^*))}^\alpha \). In particular, for global sections, this isomorphism reads as
    \begin{equation*}
        \mer{N}/\C \iso \mathcal{M}_{\proj(U(\calT^*))}^\alpha(\proj(U(\calT^*)))/\C;
    \end{equation*}
    and therefore the algebraic dimension of \(N\) coincides with the transcendence degree of the second extension.

    Since we are assuming \(k \geq 1\), we have \(\proj(U(\calT^*)) = \mathbb{P}((\C^*)^k \times \tilde U)\). This is isomorphic to \((\C^*)^{k-1} \times \tilde U\), so we may restrict the action to this subset and seek invariant meromorphic functions thereon; this is achieved by choosing a representative in \([\Lambda]_\C\) with \(\Lambda_1 = 0\) and projectivizing away the first coordinate of \((\C^*)^{k}\). We can observe that \(\tilde U = \C^{n-k} \setminus E\) and, by \Cref{minimal_codimension_no_ghosts}, \(E\) is the union of Zariski closed subsets of codimension \(> 1\) in \(\C^{n-k}\). By the E.~E.~Levi extension theorem (see e.g. \cite[Thm 1.1]{HarveyPolking1975} and the remark below), any meromorphic function on \((\C^*)^{k-1} \times \tilde U\) extends uniquely to \((\C^*)^{k-1} \times \C^{n-k} \). Thus, the problem reduces to the computation of
    \begin{equation*}
        \operatorname{tr}\deg \mer{(\C^*)^{k-1} \times \C^{n-k}}^\alpha/\C.
    \end{equation*}
    
    \paragraph{Step 2} Since \((\C^*)^{k-1} \times \C^{n-k}\) is a domain of holomorphy, any \(\alpha\)-invariant meromorphic function \(h\) on \((\C^*)^{k-1} \times \C^{n-k}\) can be written as the quotient of two entire functions \(h = \frac{f}{g}\). On \((\C^*)^{k-1} \times \C^{n-k}\) we introduce the coordinates \((z,w)\) with \(z \in (\C^*)^{k-1}\) and \(w \in \C^{n-k}\).
    We consider the universal covering \(\C^{k-1} \to (\C^*)^{k-1}\) given on each factor as \(\zeta_j \mapsto e^{2\pi i \zeta_j}\). Thus, the meromorphic lift \(\tilde h\) of \(h\) can be written as the quotient of two entire functions \(\tilde h = \frac{\tilde f}{\tilde g}\). Since \(\operatorname{Pic}(\C^{k-1} \times \C^{n-k})=\{0\}\), \(\tilde f\) and \(\tilde g\) can be chosen without common zero divisors. Similarly, we consider the lift of the action \(\alpha\) to the universal covering; this is an action which on each coordinate is written as
   
    \begin{align*}
        \tilde \alpha : \C^m \times (\C^{k-1} \times \C^{n-k}) &\longrightarrow \C^{k-1} \times \C^{n-k}\\
        t.(\zeta_j,w_\ell) &\longmapsto (\zeta_j + \langle L_j, t\rangle , e^{2\pi i \langle M_\ell,t\rangle} w_\ell),
    \end{align*}
    where \(\Lambda = (L,M)\) is a partitioning of the points in \(\Lambda\) collecting the indispensable points in \(L = (\Lambda_2,\dots, \Lambda_{k})\) and the remaining ones in \(M = (\Lambda_{k+1},\dots, \Lambda_{n})\). Notice that, having contracted the first coordinate of \((\C^*)^k\), the first indispensable point in \(L\) is \(\Lambda_2\). Since \(\tilde h\) lifts \(h\) and the latter is \(\alpha\)-invariant, \(\tilde h\) is invariant under the lifted action \(\tilde \alpha\) and also under the action of the deck transformation group \(\Gamma = \Z^{k-1}\). These actions clearly commute. Therefore, for any \(t \in \C^m\),
    \begin{equation*}
        \frac{t\cdot \tilde f}{t \cdot \tilde g} = t\cdot \tilde h = \tilde h = \frac{\tilde f}{\tilde g};
    \end{equation*}
    and thus, again using the fact that \(\operatorname{Pic}(\C^{k-1} \times \C^{n-k})=\{0\}\), there exists a holomorphic function \(\psi_t(\zeta,w)\) such that
    \begin{equation}
    \label{LVMB_invariance}
        (t\cdot \tilde f)(\zeta,w) = e^{\psi_t(\zeta,w)} \tilde f(\zeta,w) \qquad (t\cdot \tilde g)(\zeta,w) = e^{\psi_t(\zeta,w)} \tilde g(\zeta,w).
    \end{equation}

    Since $\tilde h$ is $\Gamma$-invariant, $\tilde f(\zeta+\gamma,w)\tilde g(\zeta,w)=\tilde g(\zeta+\gamma,w)\tilde f(\zeta,w)$. As $\tilde f,\tilde g$ have no common zero divisor, and since the translations $\tilde f(\cdot+\gamma), \tilde g(\cdot+\gamma)$ have the same vanishing loci of \(\tilde f,\tilde g\) respectively, \(\tilde f(\cdot+\gamma)\) and \(\tilde g(\cdot+\gamma)\) have no common zero divisor as well. Hence, there is $\varphi_\gamma\in\mathcal O(\C^{n-1})$ with
    \begin{equation}\label{7a}
    \tilde f(\zeta+\gamma,w)=e^{\varphi_\gamma(\zeta,w)}\tilde f(\zeta,w),\qquad
    \tilde g(\zeta+\gamma,w)=e^{\varphi_\gamma(\zeta,w)}\tilde g(\zeta,w),
    \end{equation}
    and $e^{\varphi_{\gamma+\gamma'}(\zeta,w)}=e^{\varphi_\gamma(\zeta+\gamma',w)+\varphi_{\gamma'}(\zeta,w)}$.

    After multiplying $\tilde f,\tilde g$ by a common unit $e^{b} \in \OO^*(\C^{n-1})$, the function $\varphi_\gamma$ may be taken independent of $w$. Indeed, the factor of automorphy $\{e^{\varphi_\gamma(\zeta,w)}\}$ defines a holomorphic line bundle $L_\varphi$ on
    $\Delta=\C^{n-1}/\Gamma\cong(\C^*)^{k-1}\times\C^{n-k}$. Let $\iota:(\C^*)^{k-1}\hookrightarrow \Delta$ be the inclusion of $\{w=0\}$ and $\rho: \Delta \to(\C^*)^{k-1}$ the projection so that $\iota\circ\rho\simeq\mathrm{id}_\Delta$ is a homotopical equivalence. Therefore, the bundle $\rho^*\iota^* L_\varphi$ has the factor of automorphy $e^{\varphi_\gamma(\zeta,0)}$ and the same Chern class as $ L_\varphi$ in $H^2(\Delta,\Z)$. Since $\Delta$ is Stein, the Oka--Grauert principle \cite[Thm.~5.3.1]{Forstneric2017} gives $ L_\varphi\cong\rho^*\iota^* L_\varphi$ as holomorphic bundles, i.e. there is a nowhere vanishing $s\in\mathcal O^*(\C^{n-1})$ with $s(\zeta+\gamma,w)=e^{\varphi_\gamma(\zeta,0)-\varphi_\gamma(\zeta,w)}s(\zeta,w)$. As $\C^{n-1}$ is contractible, $s = e^{b}$ for some holomorphic function \(b\). Multiplying both \(\tilde f\) and \(\tilde g\) by \(e^b\), one obtains new holomorphic functions (still denoted by \(\tilde f\) and \(\tilde g\)) with no common zero divisors and with ratio $\tilde h$. These functions now satisfy
    \begin{equation}
        \label{deck_invariance}
        \tilde f(\zeta+\gamma,w)=e^{\varphi_\gamma(\zeta)}\tilde f(\zeta,w),\qquad \tilde g(\zeta+\gamma,w)=e^{\varphi_\gamma(\zeta)}\tilde g(\zeta,w)
    \end{equation}
    with $\varphi_\gamma(\zeta) \dfn \varphi_\gamma(\zeta,0)$.

    \paragraph{Step 3}
    Upon multiplying both \(\tilde f\) and \(\tilde g\) by the same unit \(u \in \mathcal{O}^*(\C^{n-1})\), we claim that \(\psi_t(\zeta,w) = \psi_t\). Such a \(u\) can be found by solving a linear PDE involving the fundamental vector fields of the action \(\tilde \alpha\). Indeed, with respect to the coordinates \((\zeta,w)\), these are the holomorphic vector fields
    \begin{equation*}
        \xi_j = \sum_{\ell=1}^{k-1} L_{\ell}^j \frac{\partial}{\partial\zeta_\ell} + 2\pi i\sum_{i=1}^{n-k} M_{i}^j w_i \frac{\partial}{\partial w_i},
    \end{equation*}
    where \(L^j_\ell\) and \(M^j_i\) are the entries of the point configuration \(\Lambda\) thought to as the following block matrix
    \begin{equation}
        \label{Lambda_in_blocks}
        \Lambda =
        \begin{bmatrix}
            L\\
            M
        \end{bmatrix}
        =
        \begin{bNiceMatrix}[margin]
            \horzbar L_1 \horzbar\\
            \vdots\\
            \horzbar L_{k-1} \horzbar\\
            \hline
            \horzbar M_1 \horzbar\\
            \vdots\\
            \horzbar M_{n-k} \horzbar
        \end{bNiceMatrix}
        =
        \begin{bNiceMatrix}[margin]
            \vertbar & & \vertbar \\
            L^{1} & \cdots & L^{m} \\
            \vertbar & & \vertbar\\[1ex]
            \hline\vspace{.5ex}
            \vertbar & & \vertbar \\
            M^{1} & \cdots & M^{m} \\
            \vertbar & & \vertbar\\
        \end{bNiceMatrix}.
    \end{equation}

    Any \(\tilde\alpha\)-invariant meromorphic function \(\tilde h = \frac{\tilde f}{\tilde g}\) fulfills \(\xi_j(\tilde h) = 0\) for all \(j = 1,\dots, m\), and so \(\xi_j(\tilde f) \tilde g = \tilde f \xi_j(\tilde g)\). Since \(\tilde f\) and \(\tilde g\) have been chosen without common zero divisors, this condition forces \(\tilde f | \xi_j(\tilde f)\); hence, there exists \(q_j(\zeta,w) \in \mathcal{O}(\C^{n-1})\) such that \(\xi_j(\tilde f) = \tilde f q_j(\zeta,w)\). Similarly, we also have \(\xi_j(\tilde g) = \tilde g q_j(\zeta,w)\). As the acting group is abelian, \([\xi_i,\xi_j]=0\) for all \(i\) and \(j\); thus
    \begin{equation*}
        \xi_i(q_j) = \xi_i \left(\frac{\xi_j(\tilde f)}{\tilde f}\right) = \frac{\tilde f\,\xi_i(\xi_j(\tilde f)) - \xi_i(\tilde f) \xi_j(\tilde f)}{\tilde f^2} = \dots = \xi_j \left(\frac{\xi_i(\tilde f)}{\tilde f}\right) = \xi_j(q_i).
    \end{equation*}

    The unit \(u = e^{v}\) we need is a solution of the following PDE
    \begin{equation}
        \label{equation:unit_equation}
        \frac{\xi_j(\tilde f u)}{\tilde f u} = c_j
    \end{equation}
    for some constants \(c_j \in \C\).
    Equivalently, we can solve
    \begin{equation}
        \label{equation:log_unit_equation}
        \xi_j(v) = c_j - q_j(\zeta,w).
    \end{equation}
    Being holomorphic, each \(q_j\) and \(v\) can be expanded in power series as
    \begin{equation*}
        q_j(\zeta,w) = \sum_{J} p_{jJ}(\zeta) w^J,\quad v(\zeta,w) = \sum_{I} \nu_I(\zeta) w^I.
    \end{equation*}
    Observe that for any multi-index \(I\), one has \(w_j \frac{\partial}{\partial w_j} (w^I) = I_j w^I\); thus, the PDE
    \begin{equation*}
        \xi_j\left(\sum_I \nu_I(\zeta) w^I\right) = \xi_j(v) = c_j - q_j(\zeta,w) = c_j - \sum_{J} p_{jJ}(\zeta) w^J 
    \end{equation*}
    reduces to the following equations by comparing the corresponding terms in the power series expansions. Namely, these are
    \begin{equation}
        \label{decomposed_PDE}
        \begin{cases}
            \xi_j(\nu_0)(\zeta) = c_j - p_{j,0}(\zeta) &\qquad(\dag)\\
            \xi_j(v^\sharp) = - q_j^\sharp, &\qquad(\ddag)
        \end{cases}
    \end{equation}
    where \(v^\sharp(\zeta,w) = v(\zeta,w) - v(\zeta,0)\) and \(q_j^\sharp(\zeta,w) = q_j(\zeta,w) - q_j(\zeta,0)\).

    We start by solving the group of equations \((\ddag)\). Differentiating \cref{deck_invariance} gives
    \begin{equation*}
        q_j(\zeta+\gamma,w)=q_j(\zeta,w)+\xi_j(\varphi_\gamma)(\zeta).
    \end{equation*}
    By the equation above, \(q_j^\sharp(\zeta,w)\) is \(\Gamma\)-periodic. It therefore has a Laurent--Taylor expansion
    \begin{equation*}
        q_j^\sharp(\zeta,w)
        =\sum_{\substack{A\in\Z^{k-1}\\ I\in\mathbb N^{n-k}\setminus\{0\}}}
        b_{j,A,I}e^{2\pi i\langle A,\zeta\rangle}w^I,
    \end{equation*}
    which converges uniformly on compact subsets. On each monomial the vector fields act by
    \begin{equation*}
        \xi_j\big(e^{2\pi i\langle A,\zeta\rangle}w^I\big)
        =2\pi i\,\Phi_j(A,I)e^{2\pi i\langle A,\zeta\rangle}w^I,
        \qquad \Phi(A,I)\dfn AL+IM.
    \end{equation*}
    As observed in the proof of \Cref{compact_intersection}, the real kernel of \(\Phi\) meets \(\R^{k-1}\times\R_{\geq0}^{n-k}\) only at the origin. Intersecting this cone with the unit sphere, we therefore obtain, by compactness, a constant \(\delta>0\) such that
    \begin{equation}
        \label{sphere_equation}
        \|\Phi(A,I)\|\geq\delta\|(A,I)\|, \qquad (A,I) \in \Z^{k-1} \times \mathbb{N}^{n-k} \setminus \{0\}.
    \end{equation}
    In particular, for \(I\neq0\), the frequencies \(\Phi_j(A,I)\) cannot all vanish. We choose an index \(j=j(A,I)\) for which \(|\Phi_j(A,I)|\) is maximal, and set
    \begin{equation*}
        v_{A,I}=-\frac{b_{j,A,I}}{2\pi i\,\Phi_j(A,I)}.
    \end{equation*}
    The compatibility condition \(\xi_i(q_j)=\xi_j(q_i)\) implies \(\Phi_i(A,I)b_{j,A,I}=\Phi_j(A,I)b_{i,A,I}\); therefore, the quotient is independent of the choice of a nonzero frequency, and \(2\pi i\,\Phi_i(A,I)v_{A,I}=-b_{i,A,I}\) for every \(i\), including those for which \(\Phi_i(A,I)=0\). Moreover, by \cref{sphere_equation}
    \begin{equation*}
        |v_{A,I}|\leq\frac{\sqrt m}{2\pi\delta}
        \sum_{j=1}^m|b_{j,A,I}|,
    \end{equation*}
    since \(\|(A,I)\|\geq1\). It follows that the series
    \begin{equation*}
        v^\sharp(\zeta,w)=\sum_{\substack{A\in\Z^{k-1}\\ I\in\mathbb N^{n-k}\setminus\{0\}}}v_{A,I}e^{2\pi i\langle A,\zeta\rangle}w^I
    \end{equation*}
    also converges uniformly on compact subsets. Thus, it defines an entire \(\Gamma\)-periodic function satisfying \(\xi_j(v^\sharp)=-q_j^\sharp\).

    To solve the equations \((\dag)\) in \labelcref{decomposed_PDE}, we introduce the following change of coordinates. By the assumption \(h^{1,0}_\delbar(N) = 0\) and \Cref{holomorphic_one_forms_rank}, the matrix \(L\) in \cref{Lambda_in_blocks} has full rank equal to \(k-1\). When \(m\) exceeds \(k-1\), the columns \(L^j\)'s are not linearly independent. In this case we reparametrize \(\C^m\) so that  \(\{L'^1,\dots,L'^{k-1}\}\) is a basis of \(\Span_\C(L^1,\dots,L^m)\) while \(L'^{k},\dots, L'^{m} = 0\). We denote by \(M'^j\) the columns of the lower block after the change of basis. By performing a linear change of coordinates on \(\C^{k-1}\), we can find new coordinates \(\zeta'_j\) such that \(\sum_{\ell=1}^{k-1} L'^j_\ell \frac{\partial}{\partial \zeta_\ell} = \frac{\partial}{\partial\zeta_j'}\) for \(1\leq j \leq k-1\). With respect to these changes of coordinates, the fundamental vector fields are of the form
    \begin{equation*}
        \xi'_j =
        \begin{cases}
            \frac{\partial}{\partial\zeta_j'} + 2\pi i\sum_{i=1}^{n-k} {M'}_{i}^j w_i \frac{\partial}{\partial w_i} & 1 \leq j \leq k-1\\
            2\pi i\sum_{i=1}^{n-k} {M'}_{i}^j w_i \frac{\partial}{\partial w_i} & k \leq j\leq m.
        \end{cases}
    \end{equation*}
    Therefore, the equations \((\dag)\) in \cref{decomposed_PDE} can be written in the new coordinates as
    \begin{equation*}
        \begin{cases}
            \frac{\partial}{\partial\zeta_j'} \nu'_0 = c_j - p_{j,0}'(\zeta') & j \leq k-1\\
            0 = c_j - p_{j,0}'(\zeta') & j > k-1
        \end{cases}
        \tag{\dag}
    \end{equation*}
    For \(j > k-1\), the equation forces \(p'_{j,0}(\zeta') = c_j\). This determines the values of the corresponding \(c_j\), but also forces \(p_{j,0}'(\zeta')\) to be constant for \(j > k-1\). This is precisely what occurs here, due to the rigidity of the condition \(\xi_j(\tilde f) = \tilde f q_j\). Indeed, since for \(j > k-1\) the vector field \(\xi_j'\) acts as a scaling with respect to the primed coordinates, this condition reads as
    \begin{equation*}
        \sum_{|I| \geq b} \varepsilon_{j,I} \tilde f'_I(\zeta') w^I = \left(p'_{j,0}(\zeta') + \sum_{J >0} p'_{j,J}(\zeta') w^J \right) \left(\sum_{|I| \geq b}  \tilde f'_I(\zeta') w^I \right),
    \end{equation*}
    where \(b\) is the lowest total degree such that \(f_I'(\zeta') \not\equiv 0\) for some \(|I| = b\) and \(\varepsilon_{j,I} = 2\pi i \sum_{i=1}^{n-k} {M'}_{i}^j I_i\). Comparing all the terms of total degree \(b\), we have then
    \begin{equation*}
        \varepsilon_{j,I} \tilde f_I'(\zeta') = p'_{j,0}(\zeta') \tilde f'_I(\zeta') \quad \forall |I| = b.
    \end{equation*}
    If \(I^*\) is any multi-index of total degree \(b\) such that \(\tilde f'_{I^*}(\zeta') \not\equiv 0\), the equation above then forces \(p'_{j,0}(\zeta') = \varepsilon_{j,I^*}\), which proves the constancy of \(p'_{j,0}\) for \(j>k-1\). In turn, this implies that \(\varepsilon_{j,I} = \varepsilon_{j,I^*}\) for any \(I\) of total degree \(b\) such that \(f'_I(\zeta') \not\equiv 0\).

    On the other hand, for \(j \leq k-1\), the equations in \((\dag)\) are genuinely PDEs; to solve them we introduce the holomorphic 1-form
    \begin{equation*}
        \omega_0 = \sum_{j=1}^{k-1} (c_j - p'_{j,0}(\zeta')) d\zeta_j'
    \end{equation*}
    The condition \(\xi_j(q_i) = \xi_i(q_j)\) implies that, in the new coordinates, \(\frac{\partial p'_{j,0}}{\partial \zeta_i'} = \frac{\partial p'_{i,0}}{\partial \zeta_j'}\) for every \(1 \leq i,j \leq k-1\). This is equivalent to \(d \omega_0 = 0\). By the holomorphic Poincaré Lemma there exists a holomorphic primitive \(\nu_0'(\zeta')\). By construction, it solves \((\dag)\) for \(j \leq k-1\).

    Expressing this function in the original coordinates, we obtain \(\nu_0(\zeta) \in \mathcal{O}(\C^{k-1})\). In general, this is not \(\Gamma\)-periodic, nor is \(v(\zeta,w) = \nu_0(\zeta) + v^\sharp(\zeta,w)\). Indeed, for any \(\gamma \in \Gamma\)
    \begin{equation*}
        v(\zeta + \gamma,w) - v(\zeta,w) = (\nu_0(\zeta + \gamma) - \nu_0(\zeta)) + (v^\sharp(\zeta+\gamma,w) - v^\sharp(\zeta,w)) = \nu_0(\zeta + \gamma) - \nu_0(\zeta)
    \end{equation*}
    as the last summand vanishes because of the \(\Gamma\)-periodicity of \(v^\sharp\).

    Therefore, the holomorphic function \(u = e^v\) solves \cref{equation:unit_equation}. Integrating along the flows, and denoting the constants in the original basis again by \(c_j\), we see that \(\psi_t(\zeta,w)\) can be reduced to the pure character
    \begin{equation*}
        \psi_t = \sum_{j=1}^m c_j t_j : \C^m \to \C.
    \end{equation*}
    On the other hand, \(u \tilde f\) (and analogously \(u \tilde g\)) transform under the deck group as
    \begin{equation*}
        (u \tilde f)(\zeta + \gamma, w) = e^{v(\zeta + \gamma,w)} e^{\varphi_\gamma(\zeta)} \tilde f(\zeta,w) = e^{\nu_0(\zeta + \gamma) - \nu_0(\zeta) + \varphi_\gamma(\zeta)} (e^{v(\zeta, w)} \tilde f(\zeta,w)).
    \end{equation*}
    In particular, the factor of automorphy of \(u \tilde f\) is still independent of the coordinates \(w\).

    Henceforth, we will keep these correction factors implicit, thus we will use \(\tilde f\) and \(\tilde g\) to denote the new holomorphic functions and \(e^{\varphi_\gamma(\zeta)}\) for the corresponding factors of automorphy.

    \paragraph{Step 4}
    \Cref{LVMB_invariance,deck_invariance} constrain the power series expansion of \(\tilde f\) and \(\tilde g\). In general, these have the following form:
    \begin{equation*}
        \tilde f(\zeta,w) = \sum_{J \in \mathbb{N}^{n-k}} c_J(\zeta) w^J, \qquad \tilde g(\zeta,w) = \sum_{J \in \mathbb{N}^{n-k}} d_J(\zeta) w^J.
    \end{equation*}
    Comparing the \(J\)-th terms, \Cref{deck_invariance} implies
    \begin{equation*}
        c_J(\zeta + \gamma) = e^{\varphi_\gamma(\zeta)} c_J(\zeta), \qquad d_J(\zeta + \gamma) = e^{\varphi_\gamma(\zeta)} d_J(\zeta)
    \end{equation*}
    for every \(J\). Furthermore, \Cref{LVMB_invariance} forces the power series expansion of \(\tilde f\) to fulfill
    \begin{equation}
        \label{condition_on_power_series}
        \sum_{J \in \mathbb{N}^{n-k}} c_J(\zeta + L t) e^{2\pi i \langle J; M,t\rangle} w^J = e^{\psi_t} \sum_{J \in \mathbb{N}^{n-k}} c_J(\zeta) w^J,
    \end{equation}
    where we have adopted the notations \(\langle J;M,t\rangle \coloneqq \sum_{\ell =1}^{n-k} J_\ell \langle M_\ell, t \rangle\) and \(L t \dfn \sum_{j=1}^m L^j t_j\). Therefore, \cref{condition_on_power_series} holds if and only if, for any \(J\),
    \begin{equation}
        \label{coefficients_resonance}
        c_J(\zeta +  L t) = e^{\psi_t - 2\pi i\langle J; M,t \rangle} c_J(\zeta).
    \end{equation}

    Without loss of generality, we may assume \(\tilde f\) does not vanish identically; in this case we choose a reference multi-index \(J_0\) such that \(c_{J_0}(\zeta)\) is not identically zero. Therefore, we can rewrite \(\tilde h\) as
    \begin{equation*}
        \frac{\tilde f}{\tilde g} = \frac{\sum_{J} c_J(\zeta) w^J}{\sum_J d_J(\zeta)  w^J} = \frac{\frac{1}{c_{J_0}(\zeta)}\sum_{J} c_J(\zeta) w^J}{\frac{1}{c_{J_0}(\zeta)}\sum_J d_J(\zeta)  w^J} \eqqcolon \frac{\sum_{J} r_J(\zeta) w^J}{\sum_J s_J(\zeta)  w^J},
    \end{equation*}
    where \(r_J(\zeta) = \frac{c_J(\zeta)}{c_{J_0}(\zeta)}\) and \(s_J(\zeta) = \frac{d_J(\zeta)}{c_{J_0}(\zeta)}\). These functions have the advantage of being \(\Gamma\)-invariant; indeed,
    \begin{align*}
        r_J(\zeta + \gamma) &= \frac{e^{\varphi_\gamma(\zeta)} c_J(\zeta)}{e^{\varphi_\gamma(\zeta)}c_{J_0}(\zeta)} = r_J(\zeta),\\
        s_J(\zeta + \gamma) &= \frac{e^{\varphi_\gamma(\zeta)} d_J(\zeta)}{e^{\varphi_\gamma(\zeta)}c_{J_0}(\zeta)} = s_J(\zeta).
    \end{align*}
    This means that \(r_J\) and \(s_J\) descend to meromorphic functions on \((\C^*)^{k-1}\). We now examine how they transform under \(\tilde\alpha\). By \cref{coefficients_resonance},
    \begin{equation}
        \label{transf_law}
        r_J(\zeta+Lt)=e^{2\pi i\langle J_0-J;M,t\rangle}r_J(\zeta),
        \qquad
        s_J(\zeta+Lt)=e^{2\pi i\langle J_0-J;M,t\rangle}s_J(\zeta).
    \end{equation}
    By the hypothesis \(h^{1,0}_\delbar(N)=0\) and \Cref{holomorphic_one_forms_rank}, the map \(L:\C^m\to\C^{k-1}\) is surjective. For any right inverse \(R:\C^{k-1}\to\C^m\) of \(L\) we may substitute \(R\eta\) for \(t\); we obtain
    \begin{equation*}
        r_J(\zeta+\eta)=e^{2\pi i\langle (J_0-J);MR,\eta\rangle}r_J(\zeta) \qquad \text{for every }\eta\in\C^{k-1}.
    \end{equation*}
    If \(r_J\not\equiv0\), we can choose a point \(\zeta^*\) at which it is finite and nonzero. Thus, taking \(\eta = \zeta^* - \zeta\), one has
    \begin{equation*}
        r_j(\zeta^*) = e^{2\pi i\langle (J_0-J);MR,\zeta^* - \zeta\rangle}r_J(\zeta),
    \end{equation*}
    that is, \(r_J(\zeta)=\hat r_J e^{2\pi i\langle I_J,\zeta\rangle}\) in a neighborhood of \(\zeta^*\), where \(\hat r_J\in\C^*\) and initially \(I_J=(J_0-J)MR\in\C^{k-1}\). By the identity principle of meromorphic functions, this holds everywhere. Since this expression must respect \(\Gamma\)-periodicity, one necessarily has \(I_J\in\Z^{k-1}\). Substituting this expression in the transformation law of \(r_J\) in \Cref{transf_law}, by genericity of \(t\in\C^m\) we find
    \begin{equation}
        \label{fundamental_equation}
        \sum_{j=1}^{k-1}(I_J)_jL_j+\sum_{\ell=1}^{n-k}J_\ell M_\ell
        =\lambda,\qquad \lambda\dfn\sum_{\ell=1}^{n-k}(J_0)_\ell M_\ell.
    \end{equation}
    Equivalently, \((I_J,J-J_0)\in\FI(\Lambda)\). The same argument applies to \(s_J\), and the exponent \(I_J\) is the same whenever both coefficients \(\hat r_J\) and \(\hat s_J\) are nonzero. If \(k=1\), there are no \(\zeta\)-variables; thus, the ratios \(r_J\) and \(s_J\) are already constants; nonetheless, their transformation law gives \((J-J_0)M=0\), which is again \cref{fundamental_equation} with an empty \(I_J\).

    Dividing both numerator and denominator by \(c_{J_0}\), and still denoting them by \(\tilde f\) and \(\tilde g\), we can therefore write
    \begin{equation}
        \label{intermediate_expression}
        \tilde f(\zeta,w)=\sum_J\hat r_J e^{2\pi i\langle I_J,\zeta\rangle}w^J,
        \qquad
        \tilde g(\zeta,w)=\sum_J\hat s_J e^{2\pi i\langle I_J,\zeta\rangle}w^J,
    \end{equation}
    where \(\hat r_J,\hat s_J\) are constants, and by \Cref{compact_intersection}, the sums range over those \(J\)'s for which at least one coefficient is nonzero. All monomials \(e^{2\pi i\langle I_J,\zeta\rangle}w^J = z^{I_J}w^J\) have the same \(\alpha\)-character; dividing them by \(w^{J_0}\) gives first integral Laurent monomials.

    \paragraph{Step 5}
    For \(\lambda=J_0M\), the set of solutions of \Cref{fundamental_equation} is
    \begin{equation*}
        S_\lambda = \{ (I,J) \in \Z^{k-1} \times \mathbb{N}^{n-k} : (I,J-J_0) \in \FI(\Lambda)\}.
    \end{equation*}
    Here and below we identify \(J_0\) with \((0,J_0)\) when it occurs in \(\R^{n-1}\). By the geometry of the LVMB datum, this set is finite. Indeed, \(S_\lambda\) is a discrete subset of \((J_0 + \Rel(\Lambda)) \cap (\R^{k-1} \times \R^{n-k}_{\geq 0})\), which is compact by \Cref{compact_intersection}. This also implies that for any \(J\) there exist only finitely many multi-indexes \(I_J\) such that \((I_J,J) \in S_\lambda\).
    \paragraph{Step 6}\label{Step6}
    For any \(K \in \Z^{n-1}\) we associate the Laurent monomial
    \begin{equation*}
        \LaurMon(K) = \prod_{j=1}^{k-1} z_j^{K_{j}} \cdot \prod_{j=1}^{n-k} w_j^{K_{k-1+j}};
    \end{equation*}
    this defines a map \(\LaurMon : \Z^{n-1} \to \C[\Z^{n-1}]\), which is the canonical injection of \(\Z^{n-1}\) into its group \(\C\)-algebra \cite[II.\S 3]{Lang2002}. In particular, for any linearly independent set \(\{K^1,\dots,K^{s}\} \subset \Z^{n-1}\), its image \(\{\LaurMon(K^1),\dots, \LaurMon(K^s)\}\) is algebraically independent in \(\C(\{z_i, w_j\}_{i,j})\).
    
    This fact gives us a natural basis with respect to which to write the resonant Laurent monomials in \Cref{intermediate_expression}: indeed, since any \(K \in S_\lambda\) can be uniquely written as \(K = (0,J_0) + \kappa_K\) with \(\kappa_K \in \FI(\Lambda)\), this determines the Laurent monomial
    \begin{equation*}
        \LaurMon(K) = \LaurMon(J_0) \LaurMon(\kappa_K),
    \end{equation*}
    which is the product of the first integral Laurent monomial \(\LaurMon(\kappa_K)\) (cf. \cite{GrašičJarrahRomanovski2025}) with a fixed monomial \(\LaurMon(J_0)\) in the coordinates \(w_j\). Moreover, if \((\kappa^1,\dots, \kappa^a)\) is a \(\Z\)-basis of \(\FI(\Lambda)\), any \(\LaurMon(K)\) is of the form
    \begin{equation*}
        \LaurMon(K) = \LaurMon(J_0) \LaurMon(\kappa^1)^{\mu_1}\cdots \LaurMon(\kappa^a)^{\mu_a}
    \end{equation*}
    for some \(\mu_j \in \Z\). Therefore, \(\tilde f\) and \(\tilde g\) can be expressed as the \emph{finite} sums
    \begin{equation*}
        \tilde f = \sum_{K \in S_\lambda} \hat r_{K} \LaurMon(K), \quad \tilde g = \sum_{K \in S_\lambda} \hat s_{K} \LaurMon(K),
    \end{equation*}
    and each \(\LaurMon(K)\) is a finite product of the first integral Laurent monomials \(\LaurMon(\kappa^j)\) times a fixed \(\LaurMon(J_0)\). Therefore, we have written \(\tilde h\) as the rational expression
    \begin{equation}
        \label{final_rational_expression}
        \tilde h = \frac{\sum_{\kappa \in S_\lambda - J_0} \hat r_{\kappa + J_0} \LaurMon(\kappa)}{\sum_{{\kappa \in S_\lambda - J_0}} \hat s_{\kappa+J_0} \LaurMon(\kappa)}.
    \end{equation}

    \paragraph{Step 7}\label{Step8}
    Since the first integral Laurent monomials are the images of \(I(\Lambda)\) under \(\LaurMon\), by \cref{final_rational_expression} and the previous step,
    \begin{equation*}
        \mer{N} \iso \C(\LaurMon(\kappa): \kappa \in \FI(\Lambda)).
    \end{equation*}
    Therefore,
    \begin{equation*}
        \alg{N} = \operatorname{tr}\deg \mer{N}/\C = \operatorname{tr}\deg \C(\LaurMon(I(\Lambda))) / \C = a(\Lambda).
    \end{equation*}

    In \Cref{Dolbeault_indipspensable_points} we observed that if \(k=1,2\), the hypothesis \(h^{1,0}_\delbar(N)=0\) holds automatically. Therefore, since by \cite[Thm. 4]{Meersseman2000} \(\alg{N} \geq a(\Lambda)\), with the equality when \(k=0\), we obtain \(\alg{N} = a(\Lambda)\) for any LVMB manifold \(N\) with \(k \leq 2\).
\end{proof}

\begin{remark}
    \label{Meersseman_counterexample}
    Without the assumption \(h^{1,0}_\delbar(N) = 0\), we can no longer guarantee that the meromorphic functions arise exclusively from integral linear relations among the points in the configuration. As we see in the following example, some of them may arise also from complex linear relations.
\end{remark}
Consider the odd and balanced triangulated vector configuration \((V,\calT)\) in \(\R^2\) given by the vector configuration
\begin{equation*}
    V =
    \begin{bmatrix}
        \vertbar & & \vertbar\\
        v_1 &\cdots & v_7\\
        \vertbar & & \vertbar
    \end{bmatrix}
    =
    \begin{bmatrix}
        \sqrt{2}-1 & 0 & 0 & -\sqrt{2} & 1 & 0 & 0 \\
        \sqrt{3}-1 & 0 & 0 & 0 & 0 & -\sqrt{3} & 1
    \end{bmatrix}
\end{equation*}
and the triangulation \(\calT\) generated by the maximal cones
\begin{equation*}
    \{4,6\},\{4,7\},\{5,6\},\{5,7\}.
\end{equation*}
To produce a corresponding LVMB datum \((\Lambda,\calT^*)\), we need to select a Gale dual configuration and a period matrix. A possible choice yields
\begin{equation*}
    \Lambda =
    \begin{bmatrix}
        \horzbar \Lambda_1 \horzbar \\
        \vdots\\
        \horzbar \Lambda_7 \horzbar
    \end{bmatrix}
    =
    \begin{bmatrix}
        0 & 0\\
        1 & 0\\
        i &0 \\
        0 & 1\\
        0 & \sqrt{2}\\
        0 & i \\
        0 & i\sqrt{3}
    \end{bmatrix}
\end{equation*}
with
\begin{equation*}
    \calT^* = \bigl\{ \{1,2,3,a,b\}: a\in\{4,5\},\ b\in\{6,7\} \bigr\}.
\end{equation*}
\begin{figure}
    \centering
    \begin{tikzpicture}[
        x=2.2cm, y=1.6cm,
        line cap=round, line join=round,
        simplex/.style={draw=black!70, line width=.45pt},
        vertex/.style={circle, fill=black, inner sep=1.45pt},
        chamber/.style={fill=gray!10, pattern=north east lines,
                        pattern color=black!50},
        ray/.style={-{Stealth[length=2.1mm]}, line width=.85pt},
        ghost/.style={ray, dashed, draw=black!65},
        coordlabel/.style={font=\small, inner sep=2pt}
    ]
        \begin{scope}[xshift=3.52cm]
            \coordinate (G) at (0,0);
        \begin{scope}
            \definecolor{fanYellow}{RGB}{249,238,177}
            \definecolor{fanBlue}{RGB}{163,205,227}
            \definecolor{fanGreen}{RGB}{135,177,140}
            \definecolor{fanRose}{RGB}{177,111,125}

            \fill[fanYellow] (G) -- (1,0)          -- (0,1)          -- cycle; %
            \fill[fanBlue]   (G) -- ({-sqrt(2)},0) -- (0,1)          -- cycle; %
            \fill[fanGreen]  (G) -- ({-sqrt(2)},0) -- (0,{-sqrt(3)}) -- cycle; %
            \fill[fanRose]   (G) -- (1,0)          -- (0,{-sqrt(3)}) -- cycle; %
        \end{scope}
            \draw[ray] (G) -- ({-sqrt(2)},0)
                node[coordlabel,left] {$v_4$};
            \draw[ray] (G) -- (1,0)
                node[coordlabel,right] {$v_5$};
            \draw[ray] (G) -- (0,{-sqrt(3)})
                node[coordlabel,right] {$v_6$};
            \draw[ray] (G) -- (0,1)
                node[coordlabel,above] {$v_7$};

            \draw[ghost] (G) -- ({sqrt(2)-1},{sqrt(3)-1})
                node[coordlabel,above right] {$v_1$};
            \node[circle, draw=black!65, densely dashed, fill=white,
                minimum size=4.4pt, inner sep=0pt] at (G) {};
            \node[coordlabel,above left=2pt] at (G) {$v_2=v_3=0$};
        \end{scope}
        \node at (1.39,-2.18) {$(V,\mathcal T)$};

        \begin{scope}[shift={(3.80,-1.45)},scale=1.33]
            \coordinate (O) at (0,0);
            \coordinate (A) at (1,0);
            \coordinate (B) at ({sqrt(2)},0);
            \coordinate (C) at (0,1);
            \coordinate (D) at (0,{sqrt(3)});

            \fill[chamber] (O) -- (A) -- (C) -- cycle;
            \draw[simplex] (B) -- (O) -- (D);
            \foreach \P in {A,B}
                \foreach \Q in {C,D}
                    \draw[simplex] (\P) -- (\Q);

            \foreach \P in {O,A,B,C,D}
                \node[vertex] at (\P) {};

            \node[coordlabel,below]       at (O) {$0\;(1,2,3)$};
            \node[coordlabel,below]       at (A) {$1\;(4)$};
            \node[coordlabel,below right] at (B) {$\sqrt{2}\;(5)$};
            \node[coordlabel,left]        at (C) {$i\;(6)$};
            \node[coordlabel,above left]  at (D) {$i\sqrt{3}\;(7)$};
        \end{scope}
        \node at (4.74,-2.18) {$(\Lambda,\mathcal T^*)$};
    \end{tikzpicture}
    \caption{On the left the vector configuration \((V,\calT)\). On the right, the orthogonal projection of \((\Lambda,\calT^*)\) onto \(\{z_1 = 0\} \subset \aff^2_\C\)}.
\end{figure}

We note that there are no nontrivial affine rational relations among the rows of \(\Lambda\); thus, \(a(\Lambda) = 0\). Nevertheless, the associated LVM manifold \(N\) admits nonconstant meromorphic functions. Indeed, using the same notation described above, we obtain
\begin{equation*}
    L=
    \begin{bmatrix}
        1&0\\
        i&0
    \end{bmatrix}.
\end{equation*}
Thus, \(\operatorname{rk}_{\C}L=1\), and \Cref{holomorphic_one_forms_rank} gives \(h^{1,0}_\delbar(N)=1\). Denote by \((z_1:\dots:z_7)\) the homogeneous coordinates of \(\proj(U(\calT^*))\); since the first three correspond to the indispensable indices, we may write
\begin{equation*}
    \frac{z_2}{z_1}=e^{2\pi i\zeta_2}, \qquad \frac{z_3}{z_1}=e^{2\pi i\zeta_3}.
\end{equation*}
These are the \(\zeta\)-coordinates introduced in the proof of \Cref{algebraic_dimension_formula}. In that proof, we have introduced the lift \(\tilde \alpha\) of the action \(\alpha\) defining \(N\) as an orbit space. Since
\begin{equation*}
    \tilde\alpha : (t_1,t_2).(\zeta_2,\zeta_3) \longmapsto (\zeta_2+t_1,\zeta_3+i t_1)
\end{equation*}
the function \(\zeta_3 - i\zeta_2\) is \(\tilde \alpha\)-invariant. Moreover, as the three indispensable coordinates never vanish, the following defines a holomorphic map
\begin{align*}
    N &\longrightarrow \frac{\C}{\Z+i\Z}\\
    [\underline{z}] &\longmapsto [\zeta_3 - i\zeta_2]
\end{align*}
that is well-posed since changing the branches of the logarithms changes \(\zeta_3-i\zeta_2\) by an element of \(\Z+i\Z\). One may easily check that this map is surjective; thus, pulling back the Weierstrass \(\wp\) function gives a nonconstant meromorphic function on \(N\), represented in the logarithmic coordinates \(\zeta_2,\zeta_3\) as
\begin{equation*}
    \wp_{\Z+i\Z}(\zeta_3-i\zeta_2).
\end{equation*}
Consequently, \(\alg{N}\geq1\), although \(a(\Lambda)=0\). In particular, this function cannot be expressed rationally in the first integral Laurent monomials, since all such monomials are constant in this example.

\begin{note}
    The LVM manifold constructed as above fulfils \(k = m + 1\) and \(a(\Lambda) = 0\). This shows that the hypothesis \(h^{1,0}_\delbar (N) = 0\) in \Cref{algebraic_dimension_formula} is necessary.
\end{note}

\section{The algebraic reduction of an LVMB manifold}
As a consequence of Moishezon's results, for a compact complex manifold \(X\), the difference \(\dim_\C X - \alg{X}\) is a bimeromorphic invariant measuring how far \(X\) is from being algebraic. There exists, however, a finer bimeromorphic invariant, called the \emph{algebraic reduction}, which encodes all the birational information contained in a complex manifold. For the definition and its basic properties our main reference is \cite[\S 3]{Ueno1975}.
\begin{definition}
    Let \(X\) be a compact complex manifold. A surjective holomorphic map \(\Psi_X : \widetilde{X} \to X^\text{red}\) is an algebraic reduction of \(X\) if
    \begin{itemize}
        \item \(\widetilde{X}\) is a modification of \(X\), i.e. a compact complex manifold \(\widetilde{X}\) with a holomorphic and bimeromorphic map \(\epsilon_X : \widetilde{X} \to X\),
        \item \(X^\text{red}\) is a complete algebraic variety with \(\dim_\C X^\text{red} = \alg{X}\),
        \item \((\epsilon_X^\inv)^* \circ \Psi_X^* : \mer{X^\text{red}} \to \mer{X}\) is a ring isomorphism.
    \end{itemize}
\end{definition}
The variety \(X^{\text{red}}\) is unique up to birational equivalence. By Chow’s lemma \cite{Chow1957} and resolution of singularities \cite{Hironaka1964resolution1,Hironaka1964resolution2}, it may always be chosen smooth and projective.

A standard way of constructing a model for \(X^\text{red}\) is the following. If \(a = \alg{X} = 0\) we take \(X^\text{red} = \{\text{pt}\}\), observing that it fulfils the definition. Otherwise, we consider a set of generators \(\{g_1,\dots,g_\ell\}\) of the field extension \(\mer{X}/\C\). This set of meromorphic functions determines the meromorphic map
\begin{align*}
    \Phi^0_X : X &\longdashrightarrow \CP^{\ell}\\
    x &\longmapsto (1:g_1(x) : \dots : g_\ell(x))
\end{align*}
that is holomorphic away from a base-point locus \(B\); let \(X^\text{red}\) be the topological closure \(\overline{\Phi_0(X \setminus B)}\).  By Hironaka's resolution of indeterminacy \cite{Hironaka1975,ArocaHironakaVicente1977}, there exists a modification \(\epsilon_X : \widetilde{X} \to X\) of \(X\), together with a holomorphic map \(\Phi_X : \widetilde{X} \to X^{\text{red}}\) closing the following diagram
\begin{equation*}
    \begin{tikzcd}
        & \widetilde{X} \ar[dl, "\epsilon_X"'] \ar[dr, "\Phi_X"]\\
        X \ar[rr, "\Phi^0_X", dashed] && X^\text{red}
    \end{tikzcd}
\end{equation*}
The map \(\epsilon_X\) is a composition of blow-ups with smooth centers.

\begin{remark}
    In general, for a complete algebraic variety \(Y\), the extension \(\mer{Y}/\C\) fails to be purely transcendental; that is, for no transcendence basis \(\{\varphi_1,\dots,\varphi_a\}\) is the field extension
    \begin{equation*}
        \mer{Y}/\C(\varphi_1,\dots,\varphi_a)
    \end{equation*}
    trivial. Indeed, in general this extension is simple and algebraic \cite[Ch.~10, \S6.5]{GrauertRemmert1984}, meaning that it is generated by some element \(\theta \in \mer{Y}\) that is algebraic over \(\C(\varphi_1,\dots,\varphi_a)\). If \(\theta \in \C(\varphi_1,\dots,\varphi_a)\), then \(\Phi_Y^0\) can be taken as \(\Phi_Y^0 : Y \dashrightarrow \CP^a\). Since this map is dominant, such a manifold \(Y\) must be rational. The converse holds as well: if \(Y\) is rational then \(\mer{Y} \iso \C(\varphi_1,\dots,\varphi_a)\).
\end{remark}

The field extension we obtained in the \hyperref[Step8]{last step} of the proof of \Cref{algebraic_dimension_formula} is totally transcendental of degree \(a(\Lambda)\). The same occurs in the proof of \cite[Thm.~4]{Meersseman2000} under the assumption \(k = 0\). This immediately yields the following:
\begin{theorem}
    \label{algebraic_reduction}
    Let \(N = N(\Lambda,\calT)\) be an LVMB manifold with \(k\) indispensable indices. If \(k = 0\) or \(k \geq 1\) and \(h^{1,0}_\delbar(N)= 0\), then the algebraic reduction of \(N\) is birational to \(\CP^{a(\Lambda)}\).
\end{theorem}
In \cite[\S IV]{Meersseman2000}, Meersseman observes that the constant \(a(\Lambda)\) counts the number of independent \emph{rational} functions that an LVMB manifold admits. This is consistent with our description of the algebraic reduction of an LVMB manifold, in which these correspond to rational functions of \(\CP^{a(\Lambda)}\).

\subsection{The rational case}
One of the features characterizing LVMB manifolds is that they all admit a canonical holomorphic foliation \(\mathcal{F}\) \cite{MeerssemanVerjovsky2004,BattagliaZaffran2015}. When the rationality measure \(a(\Lambda)\) of the initial datum \((\Lambda,\calT^*)\) attains the maximal possible value \(d\), the leaves of \(\mathcal{F}\) are complex tori. In this case the leaf space \(X(\Lambda,\calT^*) \dfn N(\Lambda,\calT^*)/\mathcal{F}\) is a compact toric variety with at most finite quotient singularities and the projection \(\varpi : N(\Lambda,\calT^*) \to X(\Lambda,\calT^*)\) is a holomorphic Seifert bundle \cite{MeerssemanVerjovsky2004}.

Using \cite[Thm.~4]{Meersseman2000}, Cupit-Foutou and Zaffran show that if \(a(\Lambda) = d\) and \(k = 0\) the map \(\varpi\) realizes the algebraic reduction of \(N(\Lambda,\calT^*)\) \cite{Cupit-FoutouZaffran2007}. Analogously, \Cref{algebraic_dimension_formula} yields
\begin{proposition}
    Let \(N = N(\Lambda,\calT^*)\) be an LVMB manifold such that \(a(\Lambda) = d\), \(k \geq 1\) and \(h^{1,0}_\delbar = 0\). Then the holomorphic Seifert bundle \(\varpi : N(\Lambda,\calT^*) \to X(\Lambda,\calT^*)\) is a holomorphic algebraic reduction of \(N(\Lambda,\calT^*)\).
\end{proposition}
\begin{proof}
    By assumption, \(X \dfn X(\Lambda,\calT^*)\) is a compact toric variety of dimension \(d\). Since it is toric, \(\mer{X}/\C\) is generated by the algebraically independent rational functions \(\{\varphi_1,\dots,\varphi_d\}\). Since \(\varpi\) is holomorphic and surjective, the induced pullback \(\varpi^* : \mer{X} \to \mer{N}\) is injective. Hence, \(\{\varpi^*\varphi_1, \dots, \varpi^*\varphi_d\}\) is an algebraically independent set in \(\mer{N}\). Since we are assuming \(a(\Lambda) = d\), by \Cref{algebraic_dimension_formula}, the field extension \(\mer{N}/\varpi^*\C(X)\) is algebraic. We observe that this extension has degree 1. Indeed, any \(f \in \mer{N}\) that is algebraic over \(\varpi^*\C(X)\) fulfils
    \begin{equation*}
        (\varpi^* a_r) f^r + \cdots (\varpi^* a_1) f + (\varpi^* a_0) = 0, \qquad a_j \in \C(X).
    \end{equation*}
    Restricting this equation to a generic fiber \(\varpi^\inv(x)\), one gets
    \begin{equation*}
        a_r(x) f^r|_{\varpi^\inv(x)} + \cdots a_1(x) f|_{\varpi^\inv(x)} + a_0(x) = 0,
    \end{equation*}
    whose coefficients are complex constants. Since the fibers are  connected, this forces \(f\) to be constant along the generic fibers. Therefore, there exists an open dense subset \(U \subseteq X\) and a holomorphic function \(g_U\) on \(U\) such that \(f|_{\varpi^\inv(U)} = \varpi^* g_U\).
    
    To prove that \(g_U\) extends to a meromorphic function on \(X\), we consider the closure \(\overline{\Gamma_f} \subset N \times \CP^1\) of its graph. This is analytic since \(f\) is meromorphic, and its projection under \(\varpi \times 1_{\CP^1}\) is the closure of the graph \(\Gamma_{g_U}\). Since \(\varpi\) is a proper map, \((\varpi \times 1_{\CP^1})(\overline{\Gamma_f})\) is an analytic set in \(X \times \CP^1\) by Remmert's theorem \cite[Thm.~8.8]{Demailly2012}, therefore \(g_U\) extends to a meromorphic function \(g \in \mer{X}\) that satisfies \(\varpi^* g = f\). Thus, \(f \in \varpi^*\C(X)\) since \(X\) algebraic; and therefore \(\varpi\) is an algebraic reduction of \(N\).
\end{proof}

From this result and \cite[Prop.~3.2]{Cupit-FoutouZaffran2007}, we obtain the following
\begin{theorem}
    \label{reduction_rational_case}
    Let \(N(\Lambda,\calT^*)\) be an LVMB manifold such that \(a(\Lambda) = d\) and \(k \leq 2\). Then the holomorphic Seifert bundle \(\varpi : N(\Lambda,\calT^*) \to X\) is a holomorphic algebraic reduction of \(N(\Lambda,\calT^*)\).
\end{theorem}
\begin{remark}
    This result recovers in particular the conclusion of \Cref{algebraic_reduction} in the case \(a(\Lambda) = d\) and \(k \leq 2\). Indeed, any compact toric variety is birationally equivalent to a projective space of the same dimension.
\end{remark}

\section{The relation with the canonical leaf space}

 In the fullest generality, the projection \(N \to N/\mathcal{F}\) onto the canonical leaf space is not a holomorphic bundle; indeed when \(a(\Lambda) < d\) the leaf space \(N/\mathcal{F}\) is non-Hausdorff. Nevertheless, it still admits a nonrational toric structure \cite{BattagliaZaffran2017}. Note that the process described in \Cref{section:TheConstruction} can be reversed: to any LVMB datum \((\Lambda,\calT^*)\) one can associate a triangulated vector configuration \((V,\calT)\) \cite{Panov2025}. As described in the Introduction, this, in turn, gives rise to a triple \((\Sigma,Q,\{v_i\})\) in the following way: the quasilattice \(Q\) is taken as \(\Span_\Z(w \in V)\), the set \(\{v_i\}\) is formed by the non-ghost vectors of \(V\) and \(\Sigma\) is the complete simplicial fan encoding the triangulation datum \(\calT\). More details can be found in \cite[\S 2.1.3]{BattagliaZaffran2015}. In particular, to any LVMB datum \((\Lambda,\calT^*)\) there corresponds a quasilattice \(Q\), unique up to ambient isomorphisms. This has a geometric counterpart: the leaf space \(N(\Lambda,\calT^*)/\mathcal{F}\) of \(N(\Lambda,\calT^*)\) can be naturally identified with the complex toric quasifold \(X_\Sigma\) associated to the triple \((\Sigma,Q,\{v_i\})\). 

In the following we will show that, under the condition \(h^{1,0}(N) = 0\), the algebraic dimension of \(N(\Lambda,\calT^*)\) coincides with the rationality degree of \(Q\).

\subsection{An explicit formula for the nonrationality degree of a quasilattice}
In the Introduction we discussed how a quasilattice \(Q \subset \R^d\) can be decomposed as \(Q = L \oplus Q_\text{d}\): while \(Q_\text{d}\) is a canonical subgroup, \(L\) depends on a choice \cite{BattagliaPrato2026}. By definition, the nonrationality degree of \(Q\) is \(\deg_\text{NR}Q = \dim_\R \Span_\R Q_\text{d}\); this clearly equals \(d - \rk L\). To compute \(\deg_\text{NR}Q\), we employ vector configurations; this will also facilitate relating this quasilattice invariant to the LVMB construction. Given a quasilattice \(Q\), we consider any set of generators \(\{v_1,\dots, v_n\}\), and the vector configuration that they form \(V = (v_1, \dots, v_n)\). Note that neither \(V\) nor \(n\) can be taken canonically; in particular \(n\) can be any integer greater than or equal \(\rk Q\). Nevertheless, this gives rise to the following invariant quantities.
\begin{lemma}
    \label{rationality_measure_is_intrinsic}
    Let \(Q\) be a quasilattice in \(\R^d\) and \(\{v_1,\dots,v_n\}, \{w_1,\dots,w_{n'}\}\) be two set of generators. Consider the vector configurations \(V = (v_1,\dots,v_n)\) and \(V' = (w_1,\dots,w_{n'})\), then \(n - a(V) = n' - a(V')\) and \(n - b(V) = n' - b(V')\).
\end{lemma}
\begin{proof}
    Let \(\{e_j\}_j\) and \(\{e_j'\}_j\) be the canonical bases of \(\R^n\) and \(\R^{n'}\) respectively. Let \(\pi : \R^n \to \R^d\) be the map sending \(e_j \mapsto v_j\), and \(\pi' : \R^{n'} \to \R^d\) defined by \(e'_j \to w_j\). Then, by definition \(\Rel(V) = \ker \pi \subseteq \R^{n}\) and \(\Rel(V') = \ker\pi' \subseteq \R^{n'}\).
    We now show that, in our setting, the isomorphism provided by Schanuel's Lemma \cite[Lemma~11.28]{Faith1973} is rational. Consider the following commutative diagram with exact rows
    \begin{equation*}
        \begin{tikzcd}[column sep = large]
            0 \ar[r] & \Rel(V) \oplus \R^{n'} \ar[r,"\ker \pi \oplus \operatorname{id}"] & \R^n \oplus \R^{n'} \ar[r,"{(\pi,0)}"] & \R^d \ar[r] & 0\\
            0 \ar[r] & \ker(\pi + \pi') \ar[u] \ar[d] \ar[r] & \R^n \oplus \R^{n'} \ar[u,"F"'] \ar[d,"G"] \ar[r, "\pi + \pi'"] & \R^d \ar[u,equals] \ar[d,equals] \ar[r] & 0\\
            0 \ar[r] & \R^{n} \oplus \Rel(V') \ar[r, "\operatorname{id} \oplus \ker \pi'"] & \R^n \oplus \R^{n'} \ar[r,"{(0,\pi')}"] & \R^d \ar[r] & 0
        \end{tikzcd}
    \end{equation*}
    There, \(F,G\) are constructed as follows: since \(\pi(\Z^n) = Q = \pi'(\Z^{n'})\), for any \(e_j' \in \Z^{n'}\) there exists \(h_j = \sum_{i=1}^n h_j^i e_i\in \Z^n\) such that \(\pi'(e_j') = \pi(h_j)\). The assignment \(e_j' \mapsto h_j\) defines a linear map \(f : \R^{n'} \to \R^n\) which is rational since \(h_j^i \in \Z\) and satisfies \(\pi' = \pi \circ f\). This map induces \(F\) as the rational linear isomorphism represented by the matrix
    \begin{equation*}
        \begin{bmatrix}
            1_{\R^n} & f\\
            0 & 1_{\R^{n'}}
        \end{bmatrix};
    \end{equation*}
    moreover, it makes the diagram commute. Defining \(G\) symmetrically, we obtain a rational isomorphism \(\Phi = G \circ F^\inv\) that restricts to a rational isomorphism
    \begin{equation*}
        \varphi : \Rel(V) \oplus \R^{n'} \longrightarrow \R^{n} \oplus \Rel(V')
    \end{equation*}
    by the universal property of kernels.
    
    Therefore, if \(W\) is the rational core of \(\Rel(V)\), by maximality, \(\widetilde W \dfn W \oplus \R^{n'}\) is the rational core of \(\Rel(V) \oplus \R^{n'}\). Since \(\varphi\) is rational, its image \(\varphi(\widetilde W)\) is the rational core of \( \R^{n} \oplus \Rel(V')\). This splits as \(\varphi(\widetilde W) = \R^n \oplus W'\), with \(W'\) the rational core of \(\Rel(V')\). A fortiori, this implies
    \begin{equation*}
        n - a(V) = n - \dim W = n + n' - \dim \widetilde W = n + n' - \dim \varphi(\widetilde W) = n' - \dim W' = n' - a(V').
    \end{equation*}
    
    By an entirely analogous argument applied to the rational envelopes of \(\Rel(V)\) and \(\Rel(V')\) respectively, we also deduce \(n - b(V) = n' - b(V')\).
\end{proof}
Therefore, the quantities \(n - a(V)\) and \(n - b(V)\) are independent of the choice of generators and are invariants of the quasilattice. In fact, they are tightly linked to its geometry.

\begin{proposition}
    \label{rationality_measures}
    Let \(Q\) be a quasilattice in \(\R^d\) generated by the set \(\{v_1,\dots,v_n\}\), and let \(V = (v_1,\dots,v_n)\) be the corresponding vector configuration. Let \(Q = L \oplus Q_\text{d}\) be a splitting of \(Q\) as above, then \(\operatorname{rk} Q = n - a(V)\) and \(\operatorname{rk}L = n-b(V)\).
\end{proposition}
\begin{proof}
    For the first part we take \(S = \Z\setminus\{0\}\) as the maximal multiplicative system in \(\Z\), so that
    \begin{equation*}
        \Span_{\Q}(v_1,\dots,v_n) = Q \otimes_{\Z} \Q \iso S^{-1}Q,
    \end{equation*}
    where the isomorphism is canonically induced by the bilinear map \((x,\frac{p}{q}) \mapsto \frac{px}{q}\) via the universal property of the tensor product. Since \(Q\) is a torsion-free, finitely generated module over a PID
    \begin{align*}
        \rk Q &= \rk S^\inv Q\\
        &= \dim_{\Q} \Span_\Q(v_1,\dots,v_n)\\
        &= n - \dim_{\Q} (\Rel(V) \cap \Q^n)\\
        &= n - \dim_{\R} \ratcore{\Rel(V)}\\
        &= n-a(V).
    \end{align*}

    For the second part, we consider a new vector configuration \(V' = (w_1,\dots,w_{n'})\) which we assume to be adapted to the splitting \(Q = L \oplus Q_\text{d}\), i.e. such that \(\{w_1,\dots,w_r\}\) generates \(L\), and \(\{w_{r+1},\dots,w_{n'}\}\) generates \(Q_\text{d}\). We consider the linear map
    \begin{align*}
        \pi:\Z^{n'} &\longrightarrow Q\\
        e_i &\longmapsto w_i,
    \end{align*}
    and the subgroups
    \begin{equation*}
        D_1 = \Span_{\Z}(e_1,\dots,e_r) \subseteq \Z^{n'}, \quad D_2 = \Span_{\Z}(e_{r+1},\dots,e_{n'}) \subseteq \Z^{n'}.
    \end{equation*}
    With respect to the splitting \(D_1 \oplus D_2 = \Z^{n'}\), \(\pi\) decomposes as \(\pi_1 \oplus \pi_2\) according to the following diagram
    \begin{equation*}
        \begin{tikzcd}
            0 \ar[r] & K_1 \ar[r] \ar[d] & D_1 \ar[r,"\pi_1"] \ar[d] & L \ar[r] \ar[d] & 0\\
            0 \ar[r]& \ker\pi \ar[r]& \Z^{n'} \ar[r,"\pi"]& Q \ar[r]& 0\\
            0 \ar[r] & K_2 \ar[r] \ar[u]& D_2 \ar[r,"\pi_2"] \ar[u] & Q_{\text{d}} \ar[u] \ar[r]& 0
        \end{tikzcd}
    \end{equation*}
    where the vertical arrows are respectively the canonical coprojections and the canonical maps induced by them on kernels. Let \(\pi_{\R} : \R^{n'} \to \R^d\) be the real linear map defined by \(\pi_{\R}(e_j) = w_j\), and let \(\pi_{i,\R}\) denote its restriction to \(\Span_\R(D_i)\). The images of \(\pi_{1,\R}\) and \(\pi_{2,\R}\) are the complementary subspaces \(\Span_{\R}(L)\) and \(\Span_{\R}(Q_d)\), respectively. Since \(L\) is a lattice in \(\Span_\R(L)\), one has \(\ker\pi_{1,\R}= \Span_\R (K_1)\).
    Moreover, with respect to the identifications above, 
    \begin{equation*}
        L \iso \frac{D_1}{K_1} \iso \frac{\Z^{n'}}{K_1 \oplus D_2}.
    \end{equation*}
    Let \(T = \Span_\R(K_1 \oplus D_2) \subseteq \R^{n'}\); this is a rational space containing \(\Rel(V')\). Indeed, \(\Span_\R (K_1) = \ker \pi_{1,\R}\),
    \(\Span_\R(D_2) \supseteq \ker \pi_{2,\R}\) and
    \begin{equation}
        \label{many_inclusions}
        \Span_\R(K_1 \oplus K_2) \subseteq \ker \pi_\R = \ker\pi_{1,\R} \oplus \ker \pi_{2,\R} = \Rel(V').
    \end{equation}
    We claim that \(T\) is the rational envelope of \(\Rel(V')\). To verify this, consider any rational subspace \(T' \subseteq \R^{n'}\) containing \(\Rel(V')\); by \cref{many_inclusions} it also contains \(K_1 \oplus K_2\). Such a \(T'\) must contain \(\Span_\R(D_2)\) as well. Were this not the case, we could find a non-trivial linear functional \(f : \Span_\R(D_2) \to \R\) taking rational values on \(D_2\) and satisfying \(T' \cap D_2 \subseteq \ker f\); up to rescaling it, we may also assume \(f(D_2) \subseteq \Z\). Because of the inclusions
    \begin{equation*}
        \ker f \supseteq T' \cap \Span_\R(D_2) \supseteq \Rel(V') \cap  \Span_\R(D_2) \supseteq \ker\pi_{2,\R},
    \end{equation*}
    \(f\) factors as 
    \begin{equation*}
        \begin{tikzcd}
            \Span_\R(D_2) \ar[r, "f"] \ar[d, "p"'] & \R\\
            \frac{\Span_\R(D_2)}{\ker\pi_{2,\R}} \ar[ur, "\bar f"']
        \end{tikzcd}
    \end{equation*}
    and by construction \(\frac{\Span_\R(D_2)}{\ker\pi_{2,\R}} \iso \Span_\R(Q_\text{d}) \): this isomorphism identifies \(Q_\text{d}\) with a dense subgroup of the quotient. Evaluating \(\bar f\) on \(Q_\text{d}\), we have \(\bar f(Q_\text{d}) = (\bar f \circ p)(D_2) = f(D_2) \subseteq \Z \). On the other hand, since \(\bar f\) is continuous, \(\bar f(Q_\text{d})\) is dense in \(f(\Span_\R(D_2))\): this forces \(f(D_2) = \{0\}\), a contradiction. Therefore, \(T' \supseteq T\), whence \(T\) is the rational envelope of \(\Rel(V')\). Thus, \(\dim_{\R} T = b(V')\) and \(\operatorname{rk}L = n' - b(V') = n - b(V)\), with the second equality following from \Cref{rationality_measure_is_intrinsic}.
\end{proof}
As \(\deg_\text{NR}Q = d - \rk L\), this \namecref{rationality_measures} immediately yields a formula for the nonrationality degree of a quasilattice \(Q\).
\begin{proposition}
    \label{rationality_degree_formula}
    Let \(Q \subset \R^d\) be a quasilattice generated by a vector configuration \(V\). Then its nonrationality degree is \(\deg_\text{NR} Q = b(V) - (n-d)\).
\end{proposition}
Notably, as \(b(V) = \dim_\R \Rel(V)^\Q \geq n - d\), this quantity is non-negative, and it vanishes only if \(b(V) = n - d\). As expected from \cite[Cor.~1]{BattagliaPrato2026}, this is precisely the condition for \(Q\) to be a lattice. Indeed, a quasilattice in \(\R^d\) is a lattice if and only if \(\rk Q = d\); by \Cref{rationality_measures}, this occurs if and only if \(a(V) = n - d\) and, in turn, if and only if \(b(V) = n - d\) as well.

\subsection{The relation with the algebraic dimension}
As we have seen, to any vector configuration \(V\) we can associate a quasilattice \(Q\) whose rationality degree is \(n - b(V)\). By Gale duality, this coincides with \(a(\hat\Lambda^\R)\) for any graded Gale dual configuration of vectors \(\hat\Lambda^\R\). In turn, by the identification of \(\Rel(\hat\Lambda^\R)\) with \(\Rel_\text{aff}(\Lambda)\) and \Cref{a_s_are_equal}, this integer coincides with \(a(\Lambda)\). We can therefore state the following theorem, relating the algebraic dimension of an LVMB manifold \(N(\Lambda,\calT^*)\) with \(h^{1,0}_\delbar(N) = 0\) to the nonrationality degree of the quasilattice \(Q\) associated to the LVMB datum \((\Lambda,\calT^*)\). As before, this applies in particular to all the LVMB manifolds with less than three indispensable indices.
\begin{theorem}
    \label{rat_dimension_is_rat_measure}
    Let \(N = N(\Lambda,\calT^*)\) be the LVMB manifold associated to the LVMB datum \((\Lambda,\calT^*)\); suppose also that \(h^{1,0}_\delbar(N) = 0\). Let \(Q \subseteq \R^d\) be the quasilattice associated to the LVMB datum. Then, the algebraic dimension of \(N(\Lambda,\calT^*)\) is \(d - \deg_\text{NR}Q\).
\end{theorem}

\section{A two-parameter family of generalized Calabi--Eckmann threefolds}
LVMB manifolds appear as natural generalizations of well-known constructions in complex geometry \cite{LopezVerjovsky1997,Meersseman2000}: namely complex tori, Hopf manifolds \cite{Hopf1948} and Calabi--Eckmann manifolds \cite{Calabi-Eckmann}. Prior to the introduction of LVMB manifolds, the same classes of manifolds were generalized by J.~J.~Loeb and M.~Nicolau in \cite{LoebNicolau1996}, where the authors provide a systematic construction of complex structures on the product of odd-dimensional spheres \(S^{2n_1+1} \times S^{2n_2+1}\). Although there is no inclusion between their class of manifolds and LVMBs, all manifolds arising from the linear case of their construction belong to the LVM class.

In this section we describe a two-parameter family of LVMB manifolds that also belong to the class described by Loeb and Nicolau. While these are all diffeomorphic to \(S^3 \times S^3\), the leaf space of their canonical foliation is a family of toric quasifolds parametrized by two real numbers. These quasifolds are a variation of the one-parameter family of generalized Hirzebruch surfaces described in \cite{BattagliaPratoZaffran2019}; indeed both families contain all the Hirzebruch surfaces \(\hirz{n}\), and  in particular \(\hirz{0} = \CP^1 \times \CP^1\). Viewing Calabi--Eckmann threefolds as LVMB manifolds, it can be readily verified that the leaf space of their canonical foliation is \(\hirz{0}\) \cite{MeerssemanVerjovsky2004}. Thus, the two-parameter family of LVMB manifolds constructed below naturally constitutes a class of generalized Calabi--Eckmann threefolds.

We begin by considering the vector configuration \(V\) in \(\R^2\) consisting of
\begin{equation*}
    v_1 =
    \begin{bmatrix}
        x-1\\
        -y
    \end{bmatrix},\;
    v_2 =
    \begin{bmatrix}
        1\\
        0
    \end{bmatrix},\;
    v_3 =
    \begin{bmatrix}
        -x\\
        y
    \end{bmatrix},\;
    v_4 =
    \begin{bmatrix}
        0\\
        -1
    \end{bmatrix},\;
    v_5 =
    \begin{bmatrix}
        0\\
        1
    \end{bmatrix},
\end{equation*}
with \(x > 0\) and \(y \geq 0\). As triangulation, we take the one generated by the maximal simplexes
\begin{equation*}
    \calT = \left\langle\{2,5\}, \{3,5\}, \{3,4\}, \{2,4\}\right\rangle;
\end{equation*}
thus, \(v_1\) is the only ghost vector in the triangulated vector configuration \((V,\calT)\). For \(x = 1\) and \(y = n \in \mathbb{N}\), this triangulated vector configuration encodes the toric datum of a Hirzebruch surface \(\hirz{n}\). To construct an associated family of LVMB manifolds, we take as Gale dual configuration
\begin{equation*}
    \hat\Lambda^\R_{x,y} =
    \begin{bmatrix}
        1 & 0 & 0\\
        1 & x & 0\\
        1 & 1 & 0\\
        1 & y & 1\\
        1 & 0 & 1\\
    \end{bmatrix};
\end{equation*}
choosing \(\Pi = [1, i] \in \C^{1\times 2}\) as period matrix, the corresponding LVMB datum \((\Lambda,\calT^*)\) is given by the list of points in \(\aff^1_\C\)
\begin{equation*}
    \Lambda_{x,y} = (0,x,1,y+i,i),
\end{equation*}
with the virtual chamber
\begin{equation*}
    \calT^* = \left\langle\{1,3,4\}, \{1,2,4\}, \{1,2,5\}, \{1,3,5\}\right\rangle.
\end{equation*}
Therefore, the open subset \(U(\calT^*)\) in the LVMB construction is
\begin{equation*}
    U(\calT^*) = \C^* \times (\C^2\setminus\{0\}) \times (\C^2\setminus\{0\}) \subset \C^5.
\end{equation*}
The LVMB manifold \(N = N(\Lambda,\calT^*)\) is then obtained as the orbit space of the action
\begin{align*}
    \C \times \mathbb{P}(U(\mathcal{T}^*)) &\longrightarrow \mathbb{P}(U(\mathcal{T}^*))\\
    (t,(z_1 : \dots : z_5)) &\longmapsto (z_1: e^{2\pi i x t} z_2 : e^{2\pi i t} z_3 : e^{2\pi i (y + i)t} z_4 : e^{-2\pi t} z_5).
\end{align*}
\begin{remark}
    From the point of view of Loeb and Nicolau, \(N\) is the complex manifold generated by the action of the holomorphic flow of the vector field \(X_{x,y}\) on  \(\proj(U(\calT^*)) \iso (\C^2\setminus\{0\}) \times (\C^2\setminus\{0\})\)
    \begin{equation*}
        X_{x,y} = \sum_{j=1}^4 (\Lambda_{x,y})_{j+1} z_{j} \frac{\del}{\del z_{j}}.
    \end{equation*}
\end{remark}

From the expression of \(\hat\Lambda^\R_{x,y}\), we readily compute the nonrationality degree of the quasilattice \(Q\) associated to the LVMB datum (or equivalently to \((V,\calT)\)). Indeed, it can be verified that
\begin{equation*}
    \deg_\text{NR}Q = b(V)-3 = \dim_\Q \Span_\Q(1,x,y) - 1 \in \{0,1,2\}.
\end{equation*}
Since the initial LVMB datum has only one indispensable index, we can use \Cref{algebraic_dimension_formula} to compute the algebraic dimension of \(N\). This equals
\begin{equation*}
    \alg{N} = a(\Lambda) =2 - \deg_\text{NR}Q = 3 - \dim_\Q \Span_\Q(1,x,y) \in \{0,1,2\}.
\end{equation*}
In particular, this coincides with the number of affine rational relations among the points in \(\Lambda\), recovering \cite[Cor.~4]{LoebNicolau1996}. Moreover, by \Cref{algebraic_reduction}, \(N^\text{red}\) is birational to \(\CP^{5 - b(V)}\), that, depending on the initial datum, is a point, \(\CP^1\), or \(\CP^2\).

\printbibliography

\noindent\begin{minipage}{\linewidth}
    \begin{center}
        \small
        \rule{4cm}{.5pt}
        \bigskip
        
        Università degli Studi di Firenze, Dipartimento di Matematica e Informatica ``Ulisse Dini'',\\ V.le Morgagni 67/a, 50134 Firenze, Italia
        
        email: \href{mailto:federico.thiella@unifi.it}{federico.thiella@unifi.it}
    \end{center}
\end{minipage}

\end{document}